\documentclass[oneside,11pt]{amsart}
\usepackage{blkarray}
\usepackage{amssymb}
\usepackage{graphicx,caption,subcaption} 
\usepackage{amsmath}
\usepackage{bbm}
\usepackage{amsthm}
\usepackage{amsfonts}
\usepackage{setspace}
\usepackage{mathrsfs}
\usepackage{hyperref}
\usepackage{tensor}
\usepackage{tikz}
\usepackage{tikz-cd}
\usepackage{epstopdf}
\usepackage{enumerate}
\usepackage[enableskew]{youngtab}
\usepackage{young}
\usepackage[all,cmtip]{xy}
\usepackage[margin = .8 in]{geometry}
\usepackage{hyperref,color,polynom}

\usepackage{mdframed}
\usepackage{lipsum}

\newcommand{\zz}{\ensuremath{\mathbb{Z}}}

\newcommand{\nn}{\ensuremath{\mathbb{N}}}

\newcommand{\bMod}{\ensuremath{\textbf{Mod}}}
\newcommand{\bPMod}{\ensuremath{\textbf{PMod}}}

\newcommand{\Fl}{\ensuremath{\text{Fl}}}

\newcommand{\id}{\ensuremath{\text{id}}}

\newcommand{\Sub}{\ensuremath{\text{Sub}}}

\newcommand{\Supp}{\ensuremath{\text{Supp}}}

\newcommand{\im}{\operatorname{im}}

\newcommand{\Tor}{\operatorname{Tor}}

\newcommand{\Span}{\ensuremath{\text{span}}}

\theoremstyle{definition}

\newmdtheoremenv{frm-thm}{Theorem}
\newmdtheoremenv{frm-def}{Definition}
\newmdtheoremenv{frm-lem}{Lemma}

\newtheorem{definition}{Definition}
\newtheorem{example}{Example}

\newtheorem{proof techniques}{Proof Techniques}

\newtheorem{lemma}{Lemma}
\newtheorem{corollary}{Corollary}

\newtheorem{note}{Note}

\newtheorem{theorem}{Theorem}
\newtheorem{proposition}{Proposition}

\newtheorem*{theorem*}{Theorem}

\begin{document}

\title{The Grothendieck Ring of Certain Non-Noetherian Group-Graded Algebras via $K$-Series}
\author{Nathaniel Gallup}
\date{}
\maketitle


\begin{abstract}
With the goal of computing the Grothendieck group of certain multigraded infinite polynomial rings and the $K$-series of infinite matrix Schubert spaces, we introduce a new type of $\Gamma$-graded $k$-algebra (which we call a PDCF algebra) and a new type of graded module (a BDF module) over said algebra. Since infinite polynomial rings are not Noetherian and the modules of interest are not finitely generated, we compensate by requiring certain finiteness properties of the grading group $\Gamma$ and of the grading itself. If $R$ is a PDCF $\Gamma$-graded $k$-algebra, we prove that every projective BDF $\Gamma$-graded $R$-module is free, and that when the graded maximal ideal of $R$ is generated by a regular sequence, a BDF analog of the Hilbert Syzygy Theorem holds: every BDF $R$-module has a free resolution of BDF $R$-modules which, though not finite, has the property that each graded piece is eventually zero. We use this to show that the Grothendieck group of projective BDF $R$-modules (defined similarly to $K_0(R)$) is isomorphic to the Grothendieck group of all BDF $R$-modules (defined similarly to $G_0(R)$). We describe this Grothendieck group explicitly by using $K$-series to give an isomorphism with a certain space of formal Laurent series. Finally we give a BDF version of Serre's formula for the product in the Grothendieck group, making it into a ring. 
\end{abstract}


\section{Introduction}

In \cite{gallup2021well} we introduced a type of infinite-dimensional flag space $\Fl_\nn$ which is a $k$-functor (i.e. a functor from the category of $k$-algebras to the category of sets, for $k$ a field). This functor is covered disjointly by Schubert subfunctors $X^\circ_\sigma$ which are indexed by permutations $\sigma$ of $\nn$, and which are quotients $\text{sh}(B \sigma B)/B$ of (the fppf-sheafification of) the coset spaces $B \sigma B$. The affine analogs of these spaces, called matrix Schubert spaces, are quotients of a Cox ring with infinitely many columns (i.e. a polynomial ring $k[x_{i, j}]_n$ where $n_1, n_2, \ldots$ is an increasing sequence of natural numbers, and $i , j \in \nn$ are such that $i \leq n_j$) by an ideal $\mathfrak{d}_\sigma$ generated by certain minors. There is a torus action of $\zz^{(\nn)}$ (the abelian group of $\nn$-indexed integer sequences with finitely many nonzero entries) on the $k$-functor represented by $k[x_{i , j}]_n$ which corresponds to a multigrading of $k[x_{i , j}]_n$ by $\zz^{(\nn)}$ and since the ideal $\mathfrak{d}_\sigma$ is also multigraded, it corresponds to an equivariant sheaf on this $k$-functor. This makes $k[x_{i , j}]_n / \mathfrak{d}_\sigma$ into a $\zz^{(\nn)}$-graded $k[x_{i , j}]_n$-module. The goal of this paper is to compute the $K$-theory (i.e. Grothendieck group) of this multigraded polynomial ring. A forthcoming paper will then compute the $K$-series of $k[x_{i , j}]_n/ \mathfrak{d}_\sigma$.

Much of the classical theory in the case of finite-dimensional flag varieties (see \cite{miller2005combinatorial} for a thorough treatment) uses the fact that a polynomial ring in finitely many variables over a field is Noetherian and hence the matrix Schubert variety $k[x_{i , j} \mid 1 \leq i , j \leq n] / \mathfrak{d}_\sigma$ (with $\sigma \in S_n$) is finitely generated. But of course polynomial rings in infinitely many variables are not Noetherian and the ideal $\mathfrak{d}_\sigma$ in this case is not finitely generated. Miraculously, however, the $\zz^{(\nn)}$-grading we care about has some other finiteness properties which still allow us to define a well-behaved graded Grothendieck group containing these modules. In particular, every graded piece of $k[x_{i , j}]_n$ is finite-dimensional (such rings and modules are called \emph{finite-dimensionally graded}), the $0$th graded piece of $k[x_{i , j}]_n$ is just $k$ (such graded rings are called \emph{connected}), the graded support of $k[x_{i , j}]_n$ is $\zz_{\geq 0}^{(\nn)}$ which is a pointed submonoid of $\zz^{(\nn)}$, and for every $\alpha \in \zz^{(\nn)}$, there are only finitely many $\beta \in \zz^{(\nn)}$ which are in the support of $k[x_{i , j}]_n$ and satisfy $\beta_i \leq \alpha_i$ for every $i \in \nn$ (we call this property \emph{downward finite}). Finally, the graded maximal ideal $\bigoplus_{\alpha > 0} (k[x_{i , j}]_n)_\alpha$ is generated by the regular sequence consisting of all of the variables $x_{i , j}$. 

It turns out that these properties are all that is needed define a Grothendieck group of a (not necessarily Noetherian) graded $k$-algebra which is an integral domain. Specifically, if $\Gamma$ is an abelian group and $R$ is such a $\Gamma$-graded $k$-algebra with support $Q \subseteq \Gamma$, then we say that $R$ is \emph{pointed} if $q + q' = 0$ implies $q = q' = 0$ for $q , q' \in Q$, \emph{downward finite} if for all $g \in \Gamma$ the set $\{ h \in Q \mid h \leq_Q g \}$ is finite (the partial order on $\Gamma$ is defined by setting $Q$ to be the nonnegative elements), \emph{connected} if $R_0 = k$, and \emph{finite-dimensionally graded} if $\dim_k(R_g)$ is finite for all $g \in \Gamma$ (this condition is needed to define Hilbert series and $K$-series and is automatic if $R$ is finitely generated as a $k$-algebra). We call algebras satisfying these conditions \emph{PDCF $\Gamma$-graded $k$-algebras}. Furthermore if $M$ is a (not necessarily finitely generated) $\Gamma$-graded $R$-module, we say that it is \emph{finitely bounded below} if its $\Gamma$-support is bounded below by finitely many elements (the purpose of this condition is to guarantee closure of our category under tensor product), and we call such a module $M$ a \emph{BDF $\Gamma$-graded $R$-module} if additionally its support is downward finite and it is also finite-dimensionally graded (these conditions are automatic for finitely generated modules). 

BDF $\Gamma$-graded $R$-modules over a PDCF $\Gamma$-graded $k$-algebra share a shocking number of properties with finitely generated $\zz$-graded modules over a Noetherian $\zz_{\geq 0}$-graded $k$-algebra. Preliminarily, the category $R\text{-}\bMod^{\text{BDF}}_\Gamma$ of such modules is abelian and closed under tensor product in the category of all $R$-modules. Furthermore, recall that the Quillen–Suslin theorem (\cite{quillen1976projective}, \cite{suslin1976projective}) says that every finitely-generated projective module over a polynomial ring is free. If one is satisfied with just finitely generated \textbf{graded} modules over a positively graded ring, then the proof that all such projective modules are free is substantially easier (see \cite{minamoto2011structure}, Section 2). We use a similar technique to that found in the latter reference to prove the following analogous result for BDF modules. 

\begin{theorem*}
Every projective BDF $\Gamma$-graded $R$-module is free. 
\end{theorem*}

Recall the Hilbert Syzygy Theorem states that every finitely generated module over a Noetherian regular local ring or over a polynomial ring in finitely many variables has a finite free resolution. This is of course not true for BDF modules, but using the existence of minimal free resolutions of BDF modules, we adapt the proof given in \cite{eisenbud2013commutative} (which uses $\Tor$) to obtain the following analog.

\begin{theorem*}[BDF Hilbert Syzygy Theorem]
If $R$ is a PDCF $\Gamma$-graded $k$-algebra with graded maximal ideal generated by a regular sequence, then every BDF $\Gamma$-graded $R$-module has a resolution of BDF $\Gamma$-graded free modules with the property that each graded piece is eventually zero. 
\end{theorem*}

Classically, there are two common constructions of the Grothendieck group of a Noetherian ring. In the first, $K_0(R)$ is defined to be the Grothendieck group completion of the set of isomorphism classes of finitely generated projective $R$-modules which is a commutative monoid under direct sum. In the second, $G_0(R)$ is defined to be the free abelian group on the set of isomorphism classes of \textbf{all} finitely generated $R$-modules, modulo the relation $[M] - [N] + [L] = 0$ whenever there is a short exact sequence $0 \to M \to N \to L \to 0$. Then there is an obvious map $K_0(R) \to G_0(R)$, and when $R$ is either Noetherian regular local or a polynomial ring in finitely many variables, then the existence of a finite projective resolution gives an inverse to this map (see \cite{weibel2013k}). 

We make analogous definitions for BDF modules and prove the analogous result. Specifically we define $K^{\text{BDF}}_{\Gamma , 0}(R)$ to be the Grothendieck group completion of the set of isomorphism classes of BDF $\Gamma$-graded projective $R$-modules, and $G^{\text{BDF}}_{\Gamma , 0}(R)$ to be the free abelian group on the set of isomorphism classes of \textbf{all} BDF $\Gamma$-graded $R$-modules, modulo the same relation as before. We then prove:

\begin{theorem*}
If $R$ is a PDCF $\Gamma$-graded $k$-algebra with graded maximal ideal generated by a regular sequence, the map $K^{\text{BDF}}_{\Gamma , 0}(R) \to G^{\text{BDF}}_{\Gamma , 0}(R)$ is an isomorphism. 
\end{theorem*}

Recall that homotopy invariance tells us that the Grothendieck ring $G_0(S)$ of a polynomial ring $S = k[x_1, \ldots, x_n]$ in finitely many variables is isomorphic to $\zz$. If one is interested instead in graded modules over $S$ with the standard $\zz$-grading, then the map $G_{\zz ,0}(S) \to \zz[t , t^{-1}]$ given by sending $[M]$ to its $K$-polynomial is an isomorphism, where the latter denotes all Laurent polynomials. More generally, if $S$ is $\zz^{d}$-multigraded with graded support $Q$, the same $K$-polynomial map $G_{\zz^d ,0}(S) \to \zz[Q][\zz^d]$ is again an isomorphism, where the latter denotes all Laurent polynomials in variables indexed by $\zz^d$ (see \cite{miller2005combinatorial} Chapter 8). We give the following analogous result for BDF modules, with $\zz[Q][\zz^d]$ replaced by $\zz[[Q]]\{ \Gamma \}$, the formal Laruent series with variables indexed by $\Gamma$ still with finitely bounded below, but now also with downward finite, support. Note that in the Noetherian case, $K$-\textbf{polynomials} are honest Laurent polynomials, but because our modules are not finitely generated, we must pass to $K$-\textbf{series}. 

\begin{theorem*}
The map $K^{\text{BDF}}_{\Gamma , 0}(R) \to \zz[[Q]]\{ \Gamma \}$, which sends $[F]$ to its $K$-series, is an isomorphism. 
\end{theorem*}

Classically, the tensor product makes $K_0(R)$ into a ring (called the Grothendieck ring) and the corresponding product in $G_0(R)$ is given by Serre's formula $[M] \cdot [N] = \sum_{i = 0}^\infty (-1)^i [\Tor_i(M , N)]$ (see \cite{serre2000local}), the latter sum begin finite by the Hilbert Syzygy Theorem. For BDF modules, the tensor product does still make $K^{\text{BDF}}_{\Gamma , 0}(R)$ into a ring, but in order to interpret Serre's formula, we must make sense of an in infinite sum of elements of $G^{\text{BDF}}_{\Gamma , 0}(R)$. To that end we define a certain directed filtration of subgroups on $G^{\text{BDF}}_{\Gamma , 0}(R)$ making it into a linear topological group, and show that the sum $\sum_{i = 0}^\infty (-1)^i [\Tor_i(M , N)]$ converges. Finally we prove the following theorem:

\begin{theorem*}
The isomorphisms $\zz[[Q]]\{ \Gamma \} \cong K^{\text{BDF}}_{\Gamma , 0}(R) \cong G^{\text{BDF}}_{\Gamma , 0}(R)$ are ring isomorphisms with the product in the first given by the usual product of formal Laurent series, in the second given by the tensor product, and in the third given by Serre's formula. 
\end{theorem*} 


\section{Background}

In this section we fix notation and give some necessary background on graded rings and graded modules. See also \cite{hazrat2016graded} and \cite{nastasescu2004methods} for excellent treatments of this material (and much more). 


\subsection{Families}

In this paper, a \emph{family} of elements of a set $X$ is a function $a : I \to X$ where $I$ is an index set. We denote the image of $i \in I$ under $a$ by $a_i$ and, to distinguish families from sets, we denote this family by $(a_i \mid i \in I)$. 


\subsection{Graded Rings}

Throughout this paper \textbf{$k$ will denote a field} and \textbf{$\Gamma$ will denote an abelian group}. A \emph{$\Gamma$-grading} of a commutative $k$-algebra $R$ is a function which assigns to every element $g \in \Gamma$ a $k$-submodule $R_g \subseteq R$ such that $R = \bigoplus_{g \in \Gamma} R_g$ and for all $g , h \in \Gamma$ we have $R_g R_h \subseteq R_{g + h}$.  A \emph{$\Gamma$-graded $k$-algebra} is a $k$-algebra $R$ together with a $\Gamma$-grading of $R$. Hence we can write any element $r \in R$ as a sum $r = \sum_{g \in \Gamma} r_g$ where $r_g \in R_g$ and only finitely many of the $r_g$'s are nonzero. Sine we will never put a topology on $R$ and therefore cannot take sums with more than finitely many nonzero terms, we will often fail to mention the nonzero condition explicitly. Throughout this paper \textbf{$R$ will denote a $\Gamma$-graded $k$-algebra} which, for reasons explained in Section \ref{sec: finite dimensionally graded rings and modules}, \textbf{we will always assume is an integral domain} (unless otherwise stated). 

\begin{example}
When $\Gamma = \zz$ we obtain the usual notion of a graded ring. 
\end{example}

\begin{proposition}\label{prop: R0 is a subalgebra}
$R_0$ is a subalgebra of $R$. 
\end{proposition}

\begin{proof}
By definition $R_0$ is a $k$-subspace and we also have $R_0 \cdot R_0 \subseteq R_0$ hence $R_0$ is closed under multiplication. Furthermore we can write $1 = \sum_{g \in \Gamma} r_g$ where $r_g \in R_g$ and so for any $h \in \Gamma$ we have $r_h = r_h \cdot 1 = \sum_{g \in \Gamma} r_h r_g$. The LHS is in $R_h$, so the RHS must also be by the uniqueness of the direct sum. But $r_h r_g \in R_{h + g}$ and the latter is equal to $R_h$ if and only if $g = 0$. Therefore $r_g r_h = 0$ for all $g \neq 0$ and $r_0 r_h = r_h$. In particular we have
\begin{align*}
r_0 = r_0 \cdot 1 = \sum_{g \in \Gamma} r_0 r_g = \sum_{g \in \Gamma} r_g = 1. 
\end{align*}
Hence $1 \in R_0$ and so $R_0$ is a subalgebra as desired. 
\end{proof}


\subsection{Graded Modules}

A \emph{$\Gamma$-grading} of an $R$-module $M$ is a function which assigns to every element $g \in \Gamma$ a $k$-submodule $M_g \subseteq M$ such that $M = \bigoplus_{g \in \Gamma} M_g$ and for all $g , h \in \Gamma$ we have $R_g M_h \subseteq M_{g + h}$. In particular $R$ is itself a $\Gamma$-graded $R$-module. Elements of $M$ which are in $M_g$ for some $g \in \Gamma$ are called \emph{homogeneous of degree $g$} and if $m \in M_g$ we write $\deg(m)= g$. 

One can shift the grading of any graded module: if $M$ is a $\Gamma$-graded $R$-module and $g \in \Gamma$, then we denote by $M(g)$ the $\Gamma$-graded $R$-module which is equal to $M$ as an $R$-module but whose grading is given by $M(g)_h = M_{g + h}$. 

\begin{proposition}\label{prop: equivalent conditions of graded submodule}
Let $M$ be a $\Gamma$-graded $R$-module. The following are equivalent for a submodule $N \subseteq M$. 

\begin{enumerate}

\item $N$ is generated by homogeneous elements. 

\item $N = \bigoplus_{g \in \Gamma} N \cap M_g$.

\item For any $n \in N$ if we write $n = \sum_{g \in \Gamma} m_g$ where $m_g \in M_g$ then $m_g \in N$ for all $g \in \Gamma$. 

\end{enumerate}

\end{proposition}

\begin{proof}
((1) $\implies$ (2)). By hypothesis the family $(N \cap M_g \mid g \in \Gamma)$ is independent, and if $N$ is generated by homogeneous elements, then every element of $N$ can be written as a sum of homogeneous elements of $N$, so this family spans as well. 

((2) $\implies$ (3)). If we have $n = \sum_{g \in \Gamma} m_g$ where $m_g \in M_g$, then because the family $(N \cap M_g \mid g \in \Gamma)$ spans $M$, we can also write $n = \sum_{g \in \Gamma} n_g$ where $n_g \in N \cap M_g$. Subtraction yields $0 = \sum_{g \in \Gamma} m_g - n_g$ with $m_g - n_g \in M_g$. But because the family $(M_g \mid g \in \Gamma)$ is independent, it follows that $m_g = n_g \in N \cap M_g$ as desired. 

((3) $\implies$ (1)). Given $n \in N$, since $M = \bigoplus_{g \in \Gamma} M_g$, we can write $n = \sum_{g \in \Gamma} m_g$ where $m_g \in M_g$. Then by hypothesis we have $m_g \in N \cap M_g$, so $N$ is generated by the homogeneous elements $N \cap M_g$ for $g \in \Gamma$. 
\end{proof}

A submodule which satisfies the conditions of Proposition \ref{prop: equivalent conditions of graded submodule} is called \emph{$\Gamma$-graded} or just \emph{graded} for short. An ideal of $R$ is \emph{$\Gamma$-graded} if it is as a submodule of $R$. We now collect a technical results that we will need later. 

\begin{proposition}
Let $M$ be a $\Gamma$-graded $R$-module, and $N \subseteq M$ a graded submodule. Then there is a $\Gamma$-grading of $M / N$ given by $M / N = \bigoplus_{g \in \Gamma} (M_g + N) / N$, where we identify $(M_g + N) / N$ with the submodule of $M / N$ spanned by elements of the form $m_g + N$ for $m_g \in M_g$. 
\end{proposition}

\begin{proof}
Given $g, h \in \Gamma$, it is clear that $(M_h + N) / N$ is a $k$-subspace of $M / N$ and that $R_g \cdot (M_g + N) / N \subseteq (M_{g + h} + N) / N$. Furthermore it is also obvious that $M / N = \sum_{g \in \Gamma} (M_g + N) / N$. To show that the family $( (M_g + N) / N \mid g \in \Gamma)$ of $k$-subspaces of $M / N$ is independent, we compute:
\begin{equation*}
0 = \sum_{g \in \Gamma} m_g + N = \left[\sum_{g \in \Gamma} m_g \right] + N \implies \sum_{g \in \Gamma} m_g \in N
\end{equation*}
Because $N$ is graded, for all $g \in \Gamma$ it follows that $m_g \in N$ and hence $m_g + N = 0$ as desired. 
\end{proof}

\begin{proposition}
Let $M$ be a $\Gamma$-graded $R$-module, and $I \subseteq R$ a graded ideal. Then $IM$ is a graded submodule of $M$. 
\end{proposition}

\begin{proof}
Given $x = \sum_{j = 1}^\ell i_j m_j \in IM$, because $I$ is graded, we can write $i_j = \sum_{g \in \Gamma} (i_j)_g$ and because $M$ is graded we can write $m_j = \sum_{g \in \Gamma} (m_j)_g$. Then we can write $x = \sum_{j = 1}^\ell \sum_{g, h \in \Gamma} (i_j)_g (m_j)_h$ and $(i_j)_g (m_j)_g \in IM \cap M_{g + h}$, hence $IM$ is generated by homogeneous elements of the form $im$ where $i \in I$ is homogeneous and $m \in M$ is homogeneous, so it is a graded submodule. 
\end{proof}


\subsection{Homomorphisms of Graded Modules}

If $M$ and $N$ are two $\Gamma$-graded $R$-modules we say that an $R$-module homomorphism $\varphi: M \to N$ is \emph{$\Gamma$-graded} (or just \emph{graded}) if for all $g \in \Gamma$ we have $\varphi(M_g) \subseteq N_g$.

\begin{proposition}\label{prop: ker and im are graded}
Let $M, N$ be $\Gamma$-graded $R$-modules and $\varphi: M \to N$ be a graded $R$-module homomorphism. Then $\ker \varphi$ is a graded submodule of $M$ and $\im \varphi$ is a graded submodule of $N$. Furthermore if $y_g \in (\im \varphi)_g$ for some $g \in \Gamma$ then there exists $x_g \in M_g$ such that $\varphi(x_g) = y_g$. 
\end{proposition}

\begin{proof}
If $x \in \ker \varphi$ and we write $x = \sum_{g \in \Gamma} x_g$ then we have $0 = \varphi(x) = \sum_{g \in \Gamma} \varphi(x_g)$ and since $\varphi$ is graded, we have $\varphi(x_g) \in N_g$. Because the family $(N_g \mid g \in \Gamma)$ is independent, it follows that $\varphi(x_g) = 0$, and hence $x_g \in \ker f$ as desired. 

If $y \in \im \varphi$ and we write $y = \sum_{g \in \Gamma} y_g$, there exists some $x \in M$ such that $\varphi(x) = y$. We can also write $x = \sum_{g \in \Gamma} x_g$ and then we have $\sum_{g \in \Gamma} y_a = y = \varphi(x) = \sum_{g \in \Gamma} \varphi(x_g)$. Since $\varphi$ is graded, we have $\varphi(x_g) \in N_g$, so because the family of subspaces $(N_g \mid g \in \Gamma)$ is independent, it follows that $y_g = \varphi(x_g)$ for all $g \in \Gamma$ so $y_g \in \im \varphi$ as desired.

Taking $y = y_g$ in the previous paragraph, we obtain that $y_g = \varphi(x_g)$ and $\varphi(x_h) = 0$ for $h \neq g$ proving the final statement. 
\end{proof}


\subsection{Multigraded Polynomial Rings}

In this paper we will mainly be concerned with $\Gamma$-graded polynomial rings $S := k[ x_i \mid i \in I ]$ such that the monomials $x^\alpha$ (where $\alpha \in \zz_{\geq 0}^{(I)}$) are all homogeneous (a \emph{$\Gamma$-multigrading}). Note that the degree of a monomial in the multigraded setting is completely determined by the degrees of the variables themselves. 

\begin{example}
The main example we will be interested in throughout the paper involves the following polynomial ring which can be thought of as a Cox ring (see \cite{hu2000mori}) with infinitely many columns in the variable matrix. Suppose that $n_1, n_2, \ldots$ is any sequence of natural numbers. We denote by $k[x_{i, j}]_n$ the ring of polynomials in the variables $x_{i, j}$ where $i , j \in \nn$ and $i \leq n_j$. We picture this as a matrix $X$ with $(i , j)$ entry equal to $x_{i , j}$ whose $j$th column has zeros after the $(n_j , j)$ entry. 

Define a $\zz^{(\nn)}$-multigrading on $k[x_{i, j} \mid i , j \in \nn ]_n$ by $\deg(x_{i , j}) = e_j$, where we denote the standard basis on $\zz^{( \nn)}$ by $(e_i \mid i \in \nn)$. We denote the elements of $\zz^{(\nn)}$ by $\vec{a} = (a_1, a_2, \ldots)$ and recall that the parentheses in the notation ``$\nn$'' require that $\vec{a}$ has only finitely many nonzero entries. 
\end{example}


\section{Finite Dimensionally Graded Rings and Modules}\label{sec: finite dimensionally graded rings and modules}

In this paper we will study polynomial rings in infinitely many variables which, we note, are not finitely generated as algebras over $k$. Because of this we will need to place other finiteness conditions on our rings and modules. In particular we will study the Grothendieck group of $R$ using the Hilbert series of a $\Gamma$-graded $R$-module $M$ and will therefore need $M_g$ to be a finite-dimensional $k$-vector space for all $g \in \Gamma$. We call such a $\Gamma$-graded $R$-module \emph{finite-dimensionally graded} (some authors call this \emph{modest}). We certainly want $R$ to be a finite-dimensionally graded module over itself and in this section we give a combination of (typically easy to check) conditions on $R$ that will guarantee this.

If $M$ is a $\Gamma$-graded $R$-module, we define the \emph{$\Gamma$-support} (or just support) of $M$ to be the set $\Supp_\Gamma(M) = \{ g \in \Gamma \mid M_g \neq 0 \}$. In the case when $M = R$, we define the \emph{image monoid} $Q$ of the grading to be the submonoid of $\Gamma$ generated by the set $\Supp_\Gamma(R)$. Note that by Proposition \ref{prop: R0 is a subalgebra}, $1 \in R_0$, hence $0 \in \Supp_\Gamma(R)$. However if $R$ is \textbf{not} a domain, it is not always true that $\Supp_\Gamma(R)$ is closed under addition, as the following example shows.

\begin{example}
Give $k[x, y]$ the coarse $\zz^2$-grading with $\deg(x) = (1, 0)$ and $\deg(y) = (0 , 1)$. Then the ideal $(x y)$ is graded (it is generated by $xy$ which has degree $(1 , 1)$) and hence the quotient $k[x , y] / (xy)$ is also $\zz^2$-graded. However we have that $\left( k[x , y] / (xy) \right)_{(1 , 0)} = \Span_k( x + (xy) )$ and $\left( k[x , y] / (xy) \right)_{(0 , 1)} = \Span_k( y + (xy) )$ while $\left( k[x , y] / (xy) \right)_{(1 , 1)} = 0$ and hence $\Supp_\Gamma(R)$ is not closed under addition. 
\end{example}

\begin{proposition}\label{prop: domain implies image is monoid}
If $R$ is a $\Gamma$-graded $k$-algebra which is an \textbf{integral domain} then the image monoid $Q$ of the grading is just equal to $\Supp_\Gamma(R)$. 
\end{proposition}

\begin{proof}
We show that when $R$ is a domain, $\Supp_\Gamma(R)$ is a submonoid of $\Gamma$. It follows from Proposition \ref{prop: R0 is a subalgebra} that $0 \in \Supp_\Gamma(R)$. Given $g , h \in \Supp_\Gamma(R)$, by definition there exist nonzero $r \in R_g$ and $s \in R_h$, and hence $rs \in R_{g + h}$ but because $R$ is a domain $rs \neq 0$ so $g + h \in \Supp_\Gamma(R)$. 
\end{proof}

Proposition \ref{prop: domain implies image is monoid} is the reason we requite $R$ to be a domain throughout the paper. 


\subsection{Conditions on the Image Monoid}

We now discuss the finiteness conditions we require on the image monoid of $R$, and begin with a brief summary of some necessary monoid terminology. Recall that a \emph{monoid} consists of a set $Q$ together with a binary operation on $Q$ which is associative and has a (necessarily unique) identity. The monoid is called \emph{commutative} if the binary operation is commutative. In this paper all monoids will be commutative, hence we use ``$+$'' to denote the binary operation and ``$0$'' to denote the identity. 

A monoid $Q$ is called \emph{cancelative} if $p + u = q + u$ implies $p = q$ for all $p , q, u \in Q$. Any abelian group has an underlying commutative monoid and if $Q$ is a submonoid of an abelian group $\Gamma$, it is automatically cancelative since every element of $Q$ has an inverse in $\Gamma$. Conversely any cancelative commutative monoid can be embedded as a submonoid of an abelian group (see \cite{weibel2013k} Chapter II Section 1). 

\begin{definition}
A monoid $Q$ is called \emph{pointed} if $p + q = 0$ implies $p = q = 0$ for all $p , q \in Q$. 
\end{definition}

\begin{proposition}\label{prop: partial order on pointed cancelative commutative monoids}
Let $Q$ be a pointed submonoid of an abelian group $\Gamma$. Define a relation $\leq_Q$ on $\Gamma$ by $g \leq_Q h$ if and only if there exists $q \in Q$ such that $g + q = h$. Then $\leq_Q$ is a partial order. 
\end{proposition}

\begin{proof}
To check reflexivity, note that $g = g + 0$ and $0 \in Q$, so $g \leq_Q g$. For antisymmetry note that if $g \leq_Q h$ and $h \leq_Q g$ then there exists $p , q \in Q$ such that $b = a + p$ and $a = b + q$. Then we have $a = a + p + q$ so by cancelativity we have $0 = p + q$ and by pointedness of $Q$ we have $p= q = 0$, hence $g = h$. Finally for transitivity, if $g \leq_Q h$ and $h \leq_Q u$ then there exist $p , q \in Q$ such that $h = g + p$ and $u = b + q$, hence we have $u = g + (p + q)$ and since $Q$ is closed under addition, $p + q \in Q$, hence $g \leq_Q u$ as desired. 
\end{proof}

Note that for any $q \in Q$, we have $0 \leq_Q q$. Furthermore, if $Q$ is a pointed submonoid of an abelian group $\Gamma$ and $g \leq_Q h$ then $g + q = h$ for some $q \in Q$ and hence for any $u \in \Gamma$ we have $g + u + q = h + u$, so $g + u \leq_Q h + u$. Thus $(\Gamma , \leq_Q)$ is a partially ordered group.

\begin{definition}
We say that a subset $U$ of a poset $(P , \leq)$ is \emph{downward finite} if for all $p \in P$ the set $\{ u \in U \mid u \leq p \}$ is finite. 
If $\Gamma$ is an abelian group and $Q \subseteq \Gamma$ is a pointed submonoid, we say that a subset $U \subseteq \Gamma$ is \emph{downward finite} if it is with respect to the partial order $\leq_Q$.  
\end{definition}

\begin{proposition}
Let $\Gamma$ be an abelian group and $Q \subseteq \Gamma$ a pointed submonoid. Then $Q$ itself is downward finite if and only if for all $p \in Q$ the set $\{ q \in Q \mid q \leq_Q p \}$ is finite. 
\end{proposition}

\begin{proof}
The forward direction is clear. For the reverse, given $g \in \Gamma$ if $\{ q \in Q \mid q \leq_Q g \}$ is empty it is certainly finite and if it is nonempty then there exists $q \in Q$ such that $q \leq_Q g$ so by definition there is some $q' \in Q$ with $q + q' = g$ which implies $g \in Q$, hence by hypothesis the set $\{ q \in Q \mid q \leq_Q g \}$ is finite. 
\end{proof}

\begin{definition}
We say that a $\Gamma$-graded $k$-algebra $R$ is \emph{connected} if $R_0 = k$. 
\end{definition}

\begin{theorem}\label{thm: conditions to be pdcf}
Let $R$ be a connected $\Gamma$-graded $k$-algebra with image monoid $Q$ which is pointed and downward finite. Let $S \subseteq R$ be a set of homogeneous elements which generate $R$ as a $k$-algebra. If $S_g := R_g \cap S$ is finite for all $g \in \Gamma$, then $R$ is finite-dimensionally graded. 
\end{theorem}

\begin{proof}
Given $q \in Q$, every $r \in R_q$ can be written as a polynomial in finitely many of the elements of $S$, say $s_1, \ldots, s_n$, with coefficients in $k$. Because each element of $S$ is homogeneous, each monomial in this polynomial is also homogeneous, hence we can write $r = f_1(s_1, \ldots, s_n) + \ldots + f_m(s_1, \ldots, s_n)$ where $f_i(s_1, \ldots, s_n)$ is a $k$-linear combination of monomials each of which lies in $R_{q_i}$ for $q_1, \ldots, q_n \in Q$ distinct. We may assume without loss of generality that $q_1 = q$. By independence of the family $(R_g : g \in \Gamma)$ of $k$-subspaces of $R$, we must have $r = f_1(s_1, \ldots, s_n)$ and $f_i(s_1,\ldots, s_n) = 0$ for all $i > 1$. Suppose that $s_i \in R_{q_i}$. Given any $1 \leq i \leq n$, we may assume that $s_i$ appears with nonzero exponent in at least one of the monomials in the expansion $r = f_1(s_1, \ldots, s_n)$ (otherwise we can simply leave it out of the list $s_1, \ldots, s_n$ at the beginning). Then if $\gamma s_1^{\ell_1} \ldots s_n^{\ell_n}$ for $\gamma \in k$ is such a monomial (i.e. with $\ell_i > 0$), we have $q = \ell_1 q_1 + \ldots + \ell_n q_n$, and hence $\ell_i q_i \leq_Q q$, so since $\ell_i > 0$, we have $q_i \leq_Q q$. 

This implies that the monomials $s_1^{\ell_1} \ldots s_n^{\ell_n}$ with $q = \ell_1 q_1 + \ldots + \ell_n q_n$ and $s_i \in \bigcup_{p <_Q q} S_p$ form a $k$-spanning set for $R_q$. Since $Q$ is downward finite, the set $\{ p \in Q \mid p <_Q q \}$ is finite, and by hypothesis $S_p$ is finite for all $p \in Q$, so $\dim_k(R_q)$ is finite. 
\end{proof}

\begin{definition}
We call a $\Gamma$-graded $k$-algebra $R$ that is connected and finite-dimensionally graded and whose image monoid is pointed and downward finite a \emph{PDCF $\Gamma$-graded $k$-algebra}.
\end{definition}

\begin{example}\label{exam: k is pdcf}
Let $R$ be a PDCF $\Gamma$-graded $k$-algebra with support monoid $Q$. We make $k$ into a $\Gamma$-graded $k$-algebra by defining $k_0 = k$ and $k_g = 0$ for all $g \neq 0$. Then $k$ is PDCF because $\Supp_\Gamma(k) = \{ 0 \}$ which is obviously a pointed and downward finite submonoid of $\Gamma$, $k_0 = k$ so $k$ is connected, and $\dim_k(k_g)$ is either $0$ or $1$ so $k$ is finite-dimensionally graded.
\end{example}

\subsection{Multigradings on the Main Example}

\begin{example}
We now check that the polynomial ring $k[x_{i, j} \mid i , j \in \nn]_n$ together with the  $\zz^{(\nn)}$-grading given by $\deg(x_{i , j}) =  e_i$ is a PDCF $\zz^{(\nn)}$-graded $k$-algebra. The grading is clearly connected. The image monoid of this grading is $\zz_{\geq 0}^{( \nn)}$ which is obviously pointed and also downward finite because given $\vec{u} \in \zz_{\geq 0}^{( \nn)}$, we have that $\vec{g} \leq_{\zz_{\geq 0}^{( \nn)}} \vec{u}$ if and only if $g_i \leq u_i$ for all $i \in \nn$, and since $\vec{u}$ has only finitely many nonzero entries, there are only finitely many such $\vec{g}$ in $\zz_{\geq 0}^{( \nn)}$. Therefore the desired result follows from Theorem \ref{thm: conditions to be pdcf}. 
\end{example}


\section{Finiteness Conditions on Graded Modules}

Having defined, in the previous section, the finiteness conditions on $k$-algebra $R$ that we will need, we now turn to finiteness conditions on modules. These will be important to guarantee that the resulting category of modules is closed under tensor product. We begin with the conditions on the support. 

\begin{definition}
We say that a subset $U$ of a poset $(P,\leq)$ is \emph{finitely bounded below} if there exists $p_1, \ldots, p_n \in P$ such that for all $u \in U$ there exists $1 \leq i \leq n$ such that $p_i \leq u$. In this case we say that $U$ is \emph{bounded below by $\{p_1, \ldots, p_n \}$}. If $Q$ is a pointed submonoid of an abelian group $\Gamma$, then we say that $U \subseteq \Gamma$ is \emph{finitely bounded below} if it is with respect to the partial order $\leq_Q$. 
\end{definition}

\begin{note}
In Section 8.2 of \cite{miller2005combinatorial}, for $Q$ a submonoid of an abelian group $\Gamma$, a Laurent series $a = \sum_{g \in \Gamma} a_g \textbf{t}^g$ is said to be \emph{supported on finitely many translates of $Q$} if the support of $a$ is contained in $(h_1 + Q) \cup \ldots \cup (h_n + Q)$ for some $h_1, \ldots, h_n \in \Gamma$. As long as $Q$ is pointed, this is equivalent to the support of $a$ being a finitely bounded below subset of $\Gamma$ with respect to $\leq_Q$. We use the latter terminology because we often refer to finitely bounded below subsets in a context broader than the support of a series or module. 
\end{note}

\begin{example}
Consider the abelian group $\Gamma = \zz^\nn$ with pointed submonoid $Q = \zz_{\geq 0}^\nn$. The subset $\zz_{\geq 0}^\nn$ is finitely bounded below, since $(0 , 0 , \ldots )$ is a lower bound, but not downward finite since the set of elements less than $(1, 1, \ldots )$ contains all of the standard basis vectors $e_i = (0 , \ldots, 0, 1 , 0, \ldots )$. 
\end{example}

\begin{example}
Consider the abelian group $\zz^{(\nn)}$ with pointed submonoid $Q = \zz_{\geq 0}^{(\nn)}$. The subset $U = \{ e_i - e_{i + 1} \mid i \in \nn \}$ is downward finite because given any $(n_1, n_2, \ldots ) \in \zz^{(\nn)}$ there exists $\ell \in \nn$ such that $n_i = 0$ for all $i \geq \ell$ and then it cannot be that $e_i - e_{i + 1} \leq_Q (n_1, n_2, \ldots )$ for $i \geq \ell$. But $U$ is not finitely bounded below since if $(m_1, m_2, \ldots ) \in \zz^{(\nn)}$ were smaller than infinitely many elements of $U$, infinitely many of the $m_i$'s would have to be negative. 
\end{example}

If $R$ is a PDCF $\Gamma$-graded $k$-algebra and $M$ a $\Gamma$-graded $R$-module, we say that $M$ is \emph{finitely bounded below} or that it is \emph{downward finite} if $\Supp_\Gamma(M) \subseteq \Gamma$ satisfies the corresponding property. 

\begin{definition}
If $M$ is finitely bounded below, downward finite, and finite-dimensionally graded, we call it a \emph{BDF} $\Gamma$-graded $R$-module. We denote by $R\text{-}\textbf{Mod}_\Gamma^{\text{BDF}}$ the full subcategory of the category of $\Gamma$-graded $R$-modules consisting of the BDF modules. 
\end{definition}

\begin{note}
In Theorem 8.20 of \cite{miller2005combinatorial} it is proved that in the multigraded case, finitely generated modules are finitely bounded below. It is easy to see in the more general case of modules over a PDCF $\Gamma$-graded $k$-algebra that finitely generated modules are in fact BDF. 
\end{note}

In the rest of this section, we will prove various results about the category $R\text{-}\textbf{Mod}_\Gamma^{\text{BDF}}$, and we begin by showing that it contains the ring $R$ itself. 

\begin{proposition}
If $R$ is a PDCF $\Gamma$-graded $k$-algebra then $R$ is a BDF $\Gamma$-graded $R$-module. 
\end{proposition}

\begin{proof}
By definition of PDCF, $Q= \Supp_\Gamma(R)$ is downward finite. Furthermore for any $q \in Q$, $0 \leq_Q q$ and therefore $Q$ is finitely bounded below. Again by definition of PDCF, $R$ is a finite-dimensionally graded $k$-algebra, hence it is a finite-dimensionally graded $R$-module. 
\end{proof}

Next we show that $R\text{-}\textbf{Mod}_\Gamma^{\text{BDF}}$ is closed under shifting of the grading.

\begin{proposition} \label{prop: shifts of BDF are BDF}
If $R$ is a PDCF $\Gamma$-graded $k$-algebra and $M$ is a BDF $\Gamma$-graded $R$-module and $g \in \Gamma$ then $M(g)$ is also a BDF $\Gamma$-graded $R$-module. 
\end{proposition}

\begin{proof}
Note that $\Supp_\Gamma(M(g)) = \Supp_\Gamma(M) - g$. Therefore if $\Supp_\Gamma(M)$ is bounded below by $\{g_1, \ldots, g_n \}$, $\Supp_\Gamma(M(g))$ is bounded below by $\{g_1 - g, \ldots, g_n - g \}$. Similarly if $h \in \Gamma$, then because $\Supp_\Gamma(M)$ is downward finite, the set $\{ u \in \Supp_\Gamma(M) \mid u \leq_Q h + g \}$ is finite, hence the translation  $\{ u \in \Supp_\Gamma(M) \mid u \leq_Q h + g \} - g$ is also finite, but this set is equal to $\{ v \in \Supp_\Gamma(M(g)) \mid v \leq_Q h \}$, which is therefore also finite. Finally given any $h \in \Gamma$, $M(g)_h = M_{g + h}$ which is a finite dimensional $k$-vector space by hypothesis.
\end{proof}

Finally we show that $R\text{-}\textbf{Mod}_\Gamma^{\text{BDF}}$ is an abelian category, preceded by several necessary lemmas. 

\begin{lemma}\label{lem: closure of fbb and df sets under taking subsets}
Let $(P , \leq)$ be a poset. 
\begin{enumerate}

\item If $U \subseteq P$ is finitely bounded below and $U' \subseteq U$ then $U'$ is also finitely bounded below. 

\item If $U \subseteq P$ is downward finite and $U' \subseteq U$ then $U'$ is also downward finite. 

\end{enumerate}
\end{lemma}

\begin{proof}
\
\begin{enumerate}

\item If $U \subseteq P$ is finitely bounded below then there exists $p_1, \ldots, p_n$ such that for all $u \in U$ there exists $1 \leq i \leq n$ with $p_i \leq u$. As $U' \subseteq U$, it is true the for all $u' \in U'$ there exists $1 \leq i \leq n$ such that $p_i \leq u'$. 

\item Given $p \in P$, we have that the set $\{ u \in U \mid u \leq p \}$ is finite, but because $U' \subseteq U$ it follows that $\{ u' \in U' \mid u' \leq p \}$ is a subset of the aforementioned set and is therefore also finite. 

\end{enumerate}
\end{proof}

\begin{lemma}\label{lem: graded submodules and quotients of BDF are BDF}
If $R$ is a PDCF $\Gamma$-graded $k$-algebra, $M$ is a BDF $R$-module, and $N \subseteq M$ is a graded submodule, then $N$ and $M / N$ are also BDF $R$-modules. 
\end{lemma}

\begin{proof}
For any $g \in \Gamma$, $N_g = N \cap M_g$ is a subspace of $M_g$. Therefore $\dim_k(N_g) \leq \dim_k(M_g)$, so since $M$ is finite-dimensionally graded, so is $N$. Furthermore $\Supp_\Gamma(N) \subseteq \Supp_\Gamma(M)$, so by Lemma \ref{lem: closure of fbb and df sets under taking subsets}, $N$ is finitely bounded below and downward finite. 

Next, for any $g \in \Gamma$, there is a surjective $k$-linear map from $M_g$ to $(M / N)_g = (M_g + N)/N$ which sends $m_g \mapsto m_g + N$, the kernel of which is clearly $N_g$. Therefore $(M / N)_g$ is isomorphic as a $k$-vector space to $M_g / N_g$, which is a quotient of a finite-dimensional vector space. Hence $M / N$ is finite-dimensionally graded. Furthermore, this implies that $\Supp_\Gamma(M / N) \subseteq \Supp_\Gamma(M)$, so again by Lemma \ref{lem: closure of fbb and df sets under taking subsets}, $M / N$ is finitely bounded below and downward finite. 
\end{proof}

\begin{theorem}
Let $R$ be a PDCF $\Gamma$-graded $k$-algebra. Then $R\text{-}\textbf{Mod}_{\Gamma}^{\text{BDF}}$ is an abelian category (see \cite{stacks-project} for the definition of an abelian category).
\end{theorem}

\begin{proof}
The category of $\Gamma$-graded $R$-modules is an abelian category (see \cite{nastasescu2004methods} Section 2.2). Since $R\text{-}\textbf{Mod}_{\Gamma}^{\text{BDF}}$ is a full subcategory, it is a pre-additive, meaning its sets of morphisms are abelian groups and composition is bilinear. The zero module with the unique grading is clearly BDF. Finally finite direct sums of BDF modules are BDF since the support of a direct sum of graded modules is a union of the supports and finite union preserves being finitely bounded below and downward finite. Therefore $R\text{-}\textbf{Mod}_{\Gamma}^{\text{BDF}}$ is an additive category. 

Now if $M , N$ are BDF $\Gamma$-graded $R$-modules and $\varphi : M \to N$ is a $\Gamma$-graded $R$-module homomorphism, by Proposition \ref{prop: ker and im are graded}, $\ker \varphi \subseteq M$ and $\im \varphi \subseteq N$ are graded submodules, therefore by Lemma \ref{lem: graded submodules and quotients of BDF are BDF}, $\ker \varphi$ and $\operatorname{coker} \varphi$ are BDF $\Gamma$-graded $R$-modules. Thus $R\text{-}\textbf{Mod}_{\Gamma}^{\text{BDF}}$ is closed under kernels and cokernels and hence is a pre-abelian category. Finally, the natural map $M / \ker \varphi \to \im \varphi$ is an isomorphism in the category of $\Gamma$-graded $R$-modules (because this category is abelian) and because $R\text{-}\textbf{Mod}_{\Gamma}^{\text{BDF}}$ is a full subcategory, the inverse of this isomorphism is also in $R\text{-}\textbf{Mod}_{\Gamma}^{\text{BDF}}$, so abelianity follows. 
\end{proof}

\begin{note}
We will henceforth assume that \textbf{$R$ denotes a PDCF $\Gamma$-graded $k$-algebra that is a domain} unless otherwise stated. 
\end{note}

\begin{example}
We now show that if $\mathfrak{d}$ is any ideal of $k[x_{i , j}]_n$ generated by minors of $X$ then the quotient $k[x_{i , j}]_n / \mathfrak{d}$ is a BDF $\zz^{(\nn)}$-graded $R$-module. First of all, any generating minor of $\mathfrak{d}$ is an alternating sum of monomials which each have a degree of the form $\vec{a}$ with $a_i \in \{ 0 , 1 \}$ in the $\zz^{(\nn)}$ grading. Therefore each minor is homogeneous, so $\mathfrak{d}$ is a $\zz^{(\nn)}$-graded ideal. Hence $\mathfrak{d}$ is a graded submodule of the BDF module $k[x_{i , j}]_n$, so by Lemma \ref{lem: graded submodules and quotients of BDF are BDF}, $\mathfrak{d}$ and $k[x_{i , j}]_n / \mathfrak{d}$ are also BDF modules. 
\end{example}



\section{Tensor Product of BDF Modules}

The goal of this section is to show that $R\text{-}\bMod_\Gamma^\text{BDF}$ is not just an abelian category but is closed in the category of $R$-modules under tensor product. This will be important for two reasons in future sections: the first is that it will be needed to define minimal free resolutions of BDF modules and the second is that it will give us a product on the Grothendieck group.


\subsection{Hilbert Series}

In order to show that the support of the tensor product of two BDF modules is again finitely bounded below and downward finite, we need to give the definition of the Hilbert series of a graded module, and prove that the product of two such series is well-defined.

Given an abelian group $\Gamma$, we define $\zz^\Gamma$ to be the set of functions $a : \Gamma \to \zz$ and we identify $a$ with its generating function $a(\textbf{t}) = \sum_{g \in \Gamma} a_g \textbf{t}^g$. We would like to define the product of two such series by the following formula. 
\begin{equation}\label{eq: multiplication of power series with finitely bounded downward finite support}
\left( \sum_{g \in \Gamma} a_g \textbf{t}^g \right) \left( \sum_{h \in \Gamma} b_h \textbf{t}^h \right) = \left( \sum_{u \in \Gamma} c_u \textbf{t}^u \right)
\end{equation}
where $c_u = \sum_{g+ h = u} a_g b_h$, but this is not well-defined in general since the sum defining $c_u$ could be infinite. We restrict to a subgroup of $\zz^\Gamma$ on which this multiplication is actually well-defined. 

\begin{definition}
Define $\zz[[Q]]\{\Gamma\}$ to be the subset of elements $a(\textbf{t}) = \sum_{g \in \Gamma} a_g \textbf{t}^g \in \zz^\Gamma$ such that $\Supp_\Gamma(a) := \{ g \in \Gamma \mid a_g \neq 0 \}$ is downward finite and finitely bounded below. 
\end{definition}

We think of elements of $\zz[[Q]]\{\Gamma\}$ as analogs of formal Laurent series. We now show that multiplication of two such series is well-defined and that $\zz[[Q]]\{ \Gamma \}$ is closed under this product, but first we need several lemmas.

\begin{lemma}\label{lem: finitely bounded below and downward finite sets have the finite sum property}
Let $V, W \subseteq \Gamma$ be finitely bounded below and downward finite. Given any $u \in \Gamma$, the sets $U = \{ (g , h) \mid g \in V, h \in W,  g + h \leq_Q u \} \subseteq \Gamma \times \Gamma$, $U_1 = \{ g \in V \mid \exists h \in W, g + h \leq_Q u\}$, and $U_2 = \{ h \in W \mid \exists g \in V, g + h \leq_Q u\}$ are all finite. 
\end{lemma}

\begin{proof}
Consider the projection functions $\pi_1, \pi_2 : U \to \Gamma$ which send $(g , h)$ to $g$ and $h$ respectively. Given any $h \in \pi_2(U)$, there exists $g \in V$ such that $g + h \leq_Q u$ and since $V$ is finitely bounded below, there exists $1 \leq i \leq n$ such that $g_i \leq_Q g$, which gives $g_i + h \leq_Q g + h \leq_Q u$ and hence that $h \leq_Q u - g_i$. Since $W$ is downward finite, there are only finitely many elements of $\{ h \in W \mid h \leq_Q u - g_i \}$ and since there are only finitely many $i$'s, it follows that $\pi_2(U)$ is finite. Note that $\pi_2(U) = U_2$. A similar argument shows that $\pi_1(U) = U_1$ is finite. But $U \subseteq \pi_1(U) \times \pi_2(U)$ and since the latter is finite, so is the former.
\end{proof}

\begin{lemma}\label{lem: sum of fbb df is too}
If $V, W \subseteq \Gamma$ are finitely bounded below and downward finite, then $V + W$ is as well. 
\end{lemma}

\begin{proof}
Since $V$ and $W$ are finitely bounded below by hypothesis, there are subsets $\{ g_1, \ldots g_n \}$, $\{ h_1, \ldots, h_m \}$ such that every element of $V$ is larger than an element of the former, and every element of $W$ is larger than an element of the latter. Hence given $v + w \in V + W$ there exists $1 \leq i \leq n$ and $1 \leq j \leq m$ such that $g_i \leq_Q v$ and $h_j \leq_Q w$. Adding these inequalities together yields $g_i + h_j \leq_Q v + w$ and hence every element of $V + W$ is larger than one of the elements of the set $\{ g_i + h_j \mid 1 \leq i \leq n , 1 \leq j \leq m \}$. 

Given any $u \in \Gamma$, Lemma \ref{lem: finitely bounded below and downward finite sets have the finite sum property} implies that the set $U = \{ (g , h) \mid g \in V, h \in W,  g + h \leq_Q u \}$ is finite, but the set $\{ g \in V + W \mid g \leq_Q u \}$ is the image of $U$ under the map $(g , h) \mapsto g + h$, so this set is finite too.
\end{proof}

If $a , b \in \zz[[Q]] \{ \Gamma \}$, by Lemma \ref{lem: finitely bounded below and downward finite sets have the finite sum property}, all but finitely many of the terms in the sum $c_u = \sum_{g + h = u} a_g b_h$ from Equation \ref{eq: multiplication of power series with finitely bounded downward finite support} are zero, so the product of two elements of $\zz[[Q]]\{ \Gamma \}$ exists in $\zz^\Gamma$. We now show that this product is actually again in $\zz[[Q]]\{ \Gamma \}$. 

\begin{lemma} \label{lem: product of BDF hilbert series is BDF}
Given $a , b \in \zz[[Q]]\{ \Gamma \}$, we have that $ab \in \zz[[Q]]\{ \Gamma \}$. 
\end{lemma}

\begin{proof}
If $u \in \Supp_\Gamma(ab)$ then there must exist $g \in \Supp_\Gamma(a)$ and $h \in \Supp_\Gamma(b)$ such that $g + h = u$, i.e. $\Supp_\Gamma(ab) \subseteq  \Supp_\Gamma(a) + \Supp_\Gamma(b)$. By Lemma \ref{lem: sum of fbb df is too}, because $\Supp_\Gamma(a)$ and $\Supp_\Gamma(b)$ are finitely bounded below and downward finite, $\Supp_\Gamma(a) + \Supp_\Gamma(b)$ is too. Hence by Lemma \ref{lem: closure of fbb and df sets under taking subsets}, $\Supp_\Gamma(ab)$ is also finitely bounded below and downward finite. 
\end{proof}

It is easy to see that the multiplication defined by Equation \ref{eq: multiplication of power series with finitely bounded downward finite support} makes $\zz[[Q]]\{ \Gamma \}$ into a commutative ring with multiplicative identity equal to $1 = \textbf{t}^0$. We can now define Hilbert functions of BDF modules. 

\begin{definition}
If $M$ is a finite-dimensionally $\Gamma$-graded $R$-module, the \emph{Hilbert function} of $M$ is the function $\Gamma \to \nn$ given by $g \mapsto \dim_k(M_g)$. The associated generating function $\mathcal{H}_\Gamma(M) = \sum_{g \in \Gamma} \dim_k(M_g) \textbf{t}^g$ is called the \emph{Hilbert series} of $M$. By definition $\mathcal{H}_\Gamma(M)$ is an element of $\zz^\Gamma$. Note that if $M$ is a BDF $\Gamma$-graded $R$-module then $\mathcal{H}_\Gamma(M) \in \zz[[Q]]\{\Gamma\}$. 
\end{definition}


\subsection{Tensor Product of BDF Modules}

We are now in a position to define the $\Gamma$ grading on the tensor product of two BDF $\Gamma$-graded $R$-modules and show that the result is still BDF. 

\begin{proposition}\label{prop: tensor product grading}
Let $M$ and $N$ be BDF $\Gamma$-graded $R$-modules. Then $M \otimes_R N$ is a $\Gamma$-graded BDF $R$-module with $(M \otimes_R N)_u = \Span_k( x \otimes y \in M \otimes_R N \mid x \in M_g, y \in N_h , g + h = u )$ for $u \in \Gamma$. 
\end{proposition} 

\begin{proof}
First we show that $M \otimes_R N = \bigoplus_{u \in \Gamma} (M \otimes_R N)_u$. Note that as $k$-vector spaces we have $M = \bigoplus_{g \in \Gamma} M_g$ and $N = \bigoplus_{h \in \Gamma} N_h$. Therefore because tensor product commutes with direct sum, we have $M \otimes_k N = \bigoplus_{g , h \in \Gamma} M_g \otimes_k N_h = \bigoplus_{u \in \Gamma} \left[ \bigoplus_{g + h = u} M_g \otimes_k N_h \right] $. Therefore $M \otimes_k N$ is a $\Gamma$-graded $k$-vector space with $(M \otimes_k N )_u = \bigoplus_{g + h = u} M_g \otimes_k N_h$. Let $L$ be the $k$-subspace of $M \otimes_k N$ spanned by elements of the form $r x \otimes y - x \otimes r y$ where $r \in R$ and $x \in M$, $y \in N$ are all homogeneous. There is a $k$-linear map $\varphi: M \otimes_k N \to M \otimes_R N$ sending $x \otimes y \mapsto x \otimes y$ and clearly $L$ is contained in the kernel of $\varphi$, so we obtain a $k$-linear map $\overline{\varphi} : (M \otimes_k N) / L \to M \otimes_R N$. Conversely we define a map $\psi': M \times N \to (M \otimes_k N) / L$ by sending $(x, y) \mapsto \overline{x \otimes y}$. This is $R$-balanced because given $r = \sum_{g \in \Gamma} r_g \in R$, $x = \sum_{h \in \Gamma} x_h \in M$, and $y = \sum_{u \in \Gamma} y_u \in N$, we have that 
\begin{equation*}
\psi'(r x , y) = \overline{(rx) \otimes y} = \overline{\sum_{g , h , u \in \Gamma} r_g x_h \otimes y_u} = \overline{\sum_{g , h , u \in \Gamma} x_h \otimes r_g y_u} = \overline{x \otimes (ry)} = \psi'( x , r y).
\end{equation*}
We therefore obtain a map $\psi: M \otimes_R N \to (M \otimes_k N) /L$ which is clearly the inverse of $\overline{\varphi}$, hence $\overline{\varphi}$ is an isomorphism. 

Now if $r \in R_g$, $x \in M_h$, and $y \in N_u$, then $r x \otimes y \in M_{g + h} \otimes_k N_u$ and $x \otimes r y \in M_{g} \otimes_k N_{h + u}$, and therefore $r x \otimes y - x \otimes r y \in (M \otimes_k N )_{g + h + u}$, hence $L$ is spanned by homogeneous elements and is therefore a graded submodule. Thus $(M \otimes_k N )/ L$ is a $\Gamma$-graded $k$-vector space. Using the $k$-linear isomorphism $(M \otimes_k N ) / L \cong M \otimes_R N$ from above and the fact that image of $(M \otimes_k N )_u$ under $\varphi$ is exactly $(M \otimes_R N)_u = \Span_k( x \otimes y \in M \otimes_R N \mid x\in M_g, y \in N_h , g + h = u )$, we obtain that $M \otimes_R N = \bigoplus_{u \in \Gamma} (M \otimes_R N )_u$ so indeed $M \otimes_R N$ is a $\Gamma$-graded $R$-module. 

Now we use the Hilbert series machinery developed above to show that $M \otimes_R N$ is a BDF $\Gamma$-graded $R$-module. We begin with $M \otimes_k N$. By definition, for any $u \in \Gamma$, $(M \otimes_k N )_u = \bigoplus_{g + h = u} M_g \otimes_k N_h$ and by Lemma \ref{lem: finitely bounded below and downward finite sets have the finite sum property} since $\Supp_\Gamma(M)$ and $\Supp_\Gamma(N)$ are finitely bounded below and downward finite, the set $\{ (g , h ) \in \Supp_\Gamma(M) \times \Supp_\Gamma(N) \mid g + h = u \}$ is finite, and since $M$ and $N$ are finite-dimensionally graded, $M_g$ and $N_h$ are finite-dimensional vector spaces, so $(M \otimes_k N )_u$ is a finite-dimensional vector space of dimension $\sum_{g + h = u} \dim_k(M_g) \dim_k(N_h)$. Hence $M \otimes_k N$ is a finite-dimensionally graded $k$-vector space with Hilbert series given by $\mathcal{H}_\Gamma(M \otimes_k N) = \mathcal{H}_\Gamma(M) \mathcal{H}_\Gamma(N)$. By Lemma \ref{lem: product of BDF hilbert series is BDF} we have that $\mathcal{H}_\Gamma(M \otimes_k N) \in \zz[[Q]] \{ \Gamma \}$, and hence $\Supp_\Gamma(M \otimes_k N)$ is finitely bounded below and downward finite. Since $M \otimes_R N$ is a quotient of $M \otimes_k N$ by a graded subspace, it follows that $M \otimes_R N$ is finite-dimensionally graded too, and that $\Supp_\Gamma(M \otimes_R N)$ is finitely bounded below and downward finite as well.
\end{proof}

\begin{note}\label{not: tensor product of lots of BDF modules is BDF}
A similar proof to that of Proposition \ref{prop: tensor product grading} shows that if $R$ is a PDCF $\Gamma$-graded $k$-algebra and $M_1, \ldots, M_n$ are BDF $\Gamma$-graded $R$-modules, then if we define $(M_1 \otimes_R \ldots \otimes_R M_n)_g = \Span_k(x_1 \otimes \ldots \otimes x_n \mid x_i \in M_{h_i}, \sum_{i= 1}^n h_i = g )$ for $g \in \Gamma$, then $M_1 \otimes_R \ldots \otimes_R M_n$ is a BDF $\Gamma$-graded $R$-module.
\end{note}

Finally we record the fact that the tensor product of two graded maps is graded, which will be useful later. 

\begin{proposition}\label{prop: tensor product of graded maps is graded}
Let $M, M', N, N'$ be BDF $\Gamma$-graded $R$-modules, and let $\varphi : M \to M'$, $\psi: N \to N'$ be graded $R$-linear maps. Then $\varphi \otimes \psi : M \otimes_R N \to M' \otimes_R N'$ is also graded. 
\end{proposition}

\begin{proof}
By Proposition \ref{prop: tensor product grading}, $(M \otimes_R N)_{u}$ is spanned by elements of the form $x \otimes y$ where $\deg(x) + \deg(y) = u$. Given $x \otimes y \in M \otimes_R N$ with $x \in M_g$, $y \in N_h$ and $g + h = u$, since $\varphi$ and $\psi$ are graded, we have that $\varphi(x) \in M'_g$ and $\psi(y) \in N'_h$, and hence $[\varphi \otimes \psi](x \otimes y) = \varphi(x) \otimes \psi(y) \in (M' \otimes_R N')_{u}$. Hence indeed $\varphi \otimes \psi$ sends $(M \otimes_R N)_{u}$ to $(M' \otimes_R N')_{u}$ as desired. 
\end{proof}


\section{Free and Projective BDF Modules}

In this section we'll show that every projective BDF $\Gamma$-graded $R$-module is in fact free. We begin by describing $\Gamma$-graded free modules. To build these, we will want to take infinite direct sums of shifted copies of $R$, but we must place restrictions on such direct sums to insure that the resulting free module is still BDF. The following definition gives the necessary restriction. 

\begin{definition}
Let $Q$ be a pointed submonoid of an abelian group $\Gamma$. A family of elements $(g_i \mid i \in I)$ of $\Gamma$ is called a \emph{BDF family} if the set $\{ g_i \mid i \in I \} \subseteq \Gamma$ is downward finite and finitely bounded below and for each $g \in \Gamma$, the set $\{ i \in I \mid g_i = g \}$ is finite. If $M$ is a BDF $\Gamma$-graded $R$-module, a family of homogenous elements $(m_i \mid i \in I)$ of $M$ is called a \emph{BDF family} if each $m_i$ is nonzero and $( \deg(m_i) \mid i \in I)$ is a BDF family of elements of $\Gamma$. 
\end{definition}

We now show that the direct sum of a BDF family of shifted copies of $R$ is again BDF. 

\begin{proposition}\label{prop: certain direct sums of BDF modules are BDF}
If $(g_i \mid i \in I)$ is a BDF family of elements of $\Gamma$ then the free module $F = \bigoplus_{i \in I} R(-g_i)$ is a BDF $\Gamma$-graded $R$-module.
\end{proposition}

\begin{proof}
First note that $\Supp_\Gamma(F) = \bigcup_{i \in I} \Supp_\Gamma( R(-g_i) )$. By hypothesis the set $\{ g_i \mid i \in I \}$ is finitely bounded below, say by $\{ h_1, \ldots, h_n \}$, and we claim that $\Supp_\Gamma(F)$ is also bounded below by this set. Then given $u \in \Supp_\Gamma(F)$, there exists $i \in I$ such that $u \in \Supp_\Gamma( R(-g_i) ) = \Supp_\Gamma(R) + g_i$ and hence we have that $u - g_i \in \Supp_\Gamma(R)$ so $0 \leq_Q u - g_i$, which gives $g_i \leq_Q u$. But because $\{ g_i \mid i \in I \}$ is bounded below by $\{ h_1 \ldots, h_n \}$, then there exists $1 \leq j \leq n$ such that $h_j \leq_Q g_i$, and hence $h_j \leq_Q u$, as desired. 

We now claim that $\Supp_\Gamma(F)$ is downward finite. Given $g \in \Gamma$, we have that $\{ u \in \Supp_\Gamma(F) \mid u \leq_Q g \} = \bigcup_{i \in I } \{ u \in \Supp_\Gamma( R(-g_i) ) \mid u \leq_Q g\}$. But if $u \in \Supp_\Gamma( R(-g_i) )$, then as above $g_i \leq_Q u$ so if we also have $u \leq_Q g$, it follows that $g_i \leq_Q g$. However by hypothesis there are only finitely many $g_i$ with $g_i \leq_Q g$, say $g_{i_1} \ldots, g_{i_n}$. Hence we have that $\{ u \in \Supp_\Gamma(F) \mid u \leq_Q g \} = \bigcup_{\ell = 1}^n \{ u \in \Supp_\Gamma( R(-g_{i_\ell}) ) \mid u \leq_Q g\}$. Since $\Supp_\Gamma(R)$ is downward finite, by Proposition \ref{prop: shifts of BDF are BDF}, $\Supp_\Gamma( R(-g_{i_\ell}) )$ is also downward finite, and hence $\{ u \in \Supp_\Gamma(F) \mid u \leq_Q g \}$ is a finite union of finite sets. 

Finally we claim that $F$ is finite dimensionally graded. By definition we have that $F_h = \bigoplus_{i \in I} R(-g_i)_h = \bigoplus_{i \in I} R_{h - g_i}$. By definition of $Q = \Supp_\Gamma(R)$, $R_{h - g_j} = 0$ unless $h - g_j \geq_Q 0$ which is equivalent to $g_j \leq_Q h$. Since the set $\{ g_i \mid i \in I \}$ is downward finite, there are only finitely many such $g_j$. By hypothesis, for each of these $g_j$, there are only finitely many $i \in I$ such that $g_i = g_j$. Therefore the direct sum $\bigoplus_{i \in I} R_{h - g_i}$ is in fact a finite direct sum of finite-dimensional $k$-vector spaces as desired. 
\end{proof}

For each $i \in I$, we denote the image of $1 \in R(g_i)$ under the inclusion map $R(g_i) \hookrightarrow F$ by $e_i$ so that $(e_i \mid i \in I)$ is a homogeneous basis for $F$. Note that if we forget the grading, $F = \bigoplus_{i \in I} R(g_i)$ is equal to the free $R$-module $R^{(I)}$. 

In order to show that projective BDF modules are free, we will need to use the fact that for every BDF $\Gamma$-graded $R$-module $M$, there exists a free BDF $\Gamma$-graded $R$-module $F$ and a graded surjection $F \to M$ such that the images of a homogeneous basis for $F$ are a \textbf{minimal} BDF $R$-spanning family for $M$. That such a map exists for a given BDF family in $M$ is easy to see, as the following lemma shows.  

\begin{lemma}\label{lem: BDF surjection from a free module}
Given a BDF $\Gamma$-graded $R$-module $M$ and a BDF family $(x_i \mid i \in I)$ of elements of $M$, there exists a BDF $\Gamma$-graded free $R$-module $F = \bigoplus_{i \in I} R(-g_i)$ and a graded $R$-linear map $\varphi : F \to M$ such that the image of the basis $(e_i \mid i \in I)$ of $F$ is $(x_i \mid i \in I)$. In particular if $(x_i \mid i \in I)$ $R$-spans $M$ then this map is surjective.
\end{lemma}

\begin{proof}
For any $i \in I$, let $g_i = \deg(x_i)$. Because $(x_i \mid i \in I)$ is a BDF family, Proposition \ref{prop: certain direct sums of BDF modules are BDF} implies that $F = \bigoplus_{i \in I} R(-g_i)$ is a $\Gamma$-graded free $R$-module. Then we can define an $R$-linear map $\varphi: F \to M$ by sending $e_i \mapsto x_i$ which is graded because $\deg(e_i) = g_i = \deg(x_i)$. 
\end{proof}

We must now prove the existence of a minimal BDF $R$-spanning family in any BDF $\Gamma$-graded $R$-module $M$. Unsurprisingly first we need a BDF $\Gamma$-graded version of Nakayama's Lemma, preceded by a necessary lemma about the unique graded maximal ideal of $R$ and the graded residue field. 

\begin{lemma}\label{lem: unique graded maximal ideal}
The set $R_+ := \bigoplus_{q >_Q 0} R_q$ is the unique graded maximal ideal of $R$ and $R / R_+ \cong k$ in the category of PDCF $\Gamma$-graded $k$-algebras ($k$ is a PDCF $\Gamma$-graded $k$-algebra as in Example \ref{exam: k is pdcf}). 
\end{lemma}

\begin{proof}
Define a function $(-)_0 : R \to k$ by sending $r = \sum_{q \in Q} r_q \mapsto r_0$. This is a $\Gamma$-graded $k$-linear map, and it respects multiplication because if $r = \sum_{p \in Q} r_p$ and $s = \sum_{q \in Q} s_q$ then $(rs)_0 = \sum_{p + q = 0} r_p s_q$ and since $Q$ is pointed, if $p + q = 0$ then $p = 0$ and $q = 0$, hence $(rs)_0 = r_0 s_0$. Therefore $(-)_0$ is a graded surjective $k$-algebra homomorphism whose kernel is clearly $R_+$, so $R_+$ is a graded maximal ideal. Furthermore suppose that $I = \bigoplus_{q \in Q} R_q \cap I$ is a graded ideal of $R$. If $R_0 \cap I \neq 0$ then $I$ contains a unit and so is equal to $R$. Hence if $I$ is proper, it must be that $R_0 \cap I = 0$ and hence $I \subseteq R_+$. It follows that $R_+$ is the unique graded maximal ideal of $R$.

\end{proof}

The previous lemma shows that we can view PDCF $\Gamma$-graded $k$-algebras much like local rings and this analogy will be very advantageous. In particular it gives our BDF analog of the graded version of Nakayama's Lemma, preceded by a lemma about the well-foundedness of downward finite sets. 

\begin{lemma}\label{lem: downward finite implies well-founded}
If a subset $U$ of a poset $(P , \leq)$ is downward finite then $(U, \leq)$ is well-founded. 
\end{lemma}

\begin{proof}
Suppose that $U$ is not well-founded. Then there exists a strictly infinite descending chain $u_1 > u_2 > \ldots$ in $U$ and so clearly the set $\{ u \in U \mid u \leq u_1 \}$ is not finite, a contradiction. 
\end{proof}

\begin{lemma}[BDF $\Gamma$-graded Nakayama's Lemma]
If $R$ is a PDCF $\Gamma$-graded $k$-algebra, $M$ is a BDF $\Gamma$-graded $R$-module, and $R_+ M = M$, then $M = 0$. 
\end{lemma}

\begin{proof}
Suppose that $M \neq 0$. Then $\Supp_\Gamma(M)$ is non-empty and downward finite, so it has a minimal element with respect to $\leq_Q$ by Lemma \ref{lem: downward finite implies well-founded}, call it $h$. Note that $R_+ M = \sum_{q >_Q 0, g \in \Supp_\Gamma(M)} R_q M_g$, but $R_q M_g \subseteq M_{g + q}$ and since $q >_Q 0$ it follows that $g + q >_Q g$. Hence every element of $\Supp_\Gamma(R_+ M)$ must be strictly larger than some element of $\Supp_\Gamma(M)$. In particular, $h \notin \Supp_\Gamma(R_+ M)$, which shows that $R_+ M \neq M$. 
\end{proof}

We need one more technical lemma before we can discuss minimal $R$-spanning families. It describes how tensoring a BDF module $M$ with $k \cong R / R_+$ as an $R$-module is the same as quotienting $M$ by $R_+ M$. The graded $k$-algebra homomorphism $(-)_0: R \to k$ from the proof of Lemma \ref{lem: unique graded maximal ideal} makes $k$ into a $\Gamma$-graded $R$-module with multiplication $R \times k \to k$ given by sending $r , \gamma \mapsto r_0 \gamma$. Then $k$ is in fact a BDF $\Gamma$-graded $R$-module because $\Supp_\Gamma(k) = \{ 0 \}$ is clearly finitely bounded below and downward finite. 

\begin{lemma}\label{lem: iso between tensoring with k and modding out by M plus}
For any BDF $\Gamma$-graded $R$-module $M$ there is a graded $R$-linear isomorphism $M \otimes_R k \cong M / R_+ M$. 
\end{lemma}

\begin{proof}
First we define a map $\varphi: M \times k \to M / R_+ M$ which sends $x , \gamma \mapsto \gamma \overline{x}$. This map is obviously $k$-bilinear and it is $R$-balanced because given $r = \sum_{q \in Q} r_q \in R$ we have 
\begin{equation*}
\varphi(rx , \gamma ) = \gamma \overline{ \sum_{q \in Q} r_q x} \overset{(1)}{=} \gamma \overline{r_0 x} = \gamma r_0 \overline{x} = \varphi(x , r_0 \gamma) = \varphi(x , r \gamma)
\end{equation*}
where (1) follows because $\sum_{q \in Q} r_q x = r_0 x + \sum_{q > 0} r_q x$ and the latter term is in $R_+ M$. Hence we obtain an $R$-linear map $\tilde{\varphi} : M \otimes_R k \to M / R_+ M$. The degree of $x \otimes \gamma$ is $\deg(x) + \deg(\gamma) =  \deg(x) = \deg(\overline{x})$ so this map is indeed graded. On the other hand we define a map $\psi : M \to M \otimes_R k$ by sending $x \mapsto x \otimes 1$ which is clearly $R$-linear and the kernel of $\psi$ contains $R_+ M$ since given $r x$ where $r \in R_q$ for $q > 0$ and $x \in M_g$ we have $\psi(rx) = rx \otimes 1 = x \otimes r \cdot 1$, but $r \cdot 1 = 0$ since $r \in R_+$, hence $\psi(rx) = 0$. Since $R_+ M$ is spanned by elements of the same form as $rx$, it follows that $\ker \psi \subseteq R_+ M$. Hence we obtain a map $\overline{\psi} : M / R_+ M \to M \otimes_R k$. The maps $\tilde{\varphi}$ and $\overline{\psi}$ are obviously inverses of each other.
\end{proof}

We are now finally in a position to discuss minimal $R$-spanning families. The following proposition describes them in detail. 

\begin{proposition}\label{prop: k-gen for quotient gives r-gen for module}
Let $M$ be a BDF $\Gamma$-graded $R$-module, $(x_i \mid i \in I)$ be a BDF family of elements of $M$ with $\deg(x_i) = g_i$, $F = \bigoplus_{i \in I} R(-g_i)$, and $\varphi: F \to M$ be the $R$-linear map sending $e_i \mapsto x_i$. 

\begin{enumerate}

\item The following are equivalent:

\begin{enumerate}

\item $(\overline{x_i} \mid i \in I)$ $k$-spans $M / R_+ M$.

\item $(x_i \mid i \in I)$ $R$-spans $M$. 

\item $\im \varphi = M$. 

\end{enumerate}

\item If $(\overline{x}_i \mid i \in I)$ is $k$-independent then $\ker \varphi \subseteq R_+ F$. 

\item The following are equivalent:

\begin{enumerate}

\item $(\overline{x_i} \mid i \in I)$ is a $k$-basis of $M / R_+ M$.

\item $(x_i \mid i \in I)$ is a minimal $R$-spanning family for $M$ in the sense that no proper subfamily $(x_j \mid j \in J)$ with $J \subsetneq I$ $R$-spans $M$.

\item $(x_i \mid i \in I)$ $R$-spans $M$ and $\ker \varphi \subseteq R_+ F$. 

\end{enumerate}

\end{enumerate}
\end{proposition}

\begin{proof}
\

\begin{enumerate}

\item (b) $\iff$ (c) is obvious. To prove (a) $\implies$ (c), suppose that $(\overline{_i} \mid i \in I)$ $k$-spans $M / R_+ M$. Let $C$ be the cokernel of $\varphi$ so that we obtain an exact sequence $F \to M \to C \to 0$. Since the functor $- \otimes_R k$ is right exact, we obtain an exact sequence $F \otimes_R k \to M \otimes_R k \to C \otimes_R k \to 0$. We claim that $\varphi \otimes \id$ is surjective. Indeed, by the proof of Lemma \ref{lem: iso between tensoring with k and modding out by M plus}, under the isomorphism $M / R_+ M \cong M \otimes_R k$, the element $\overline{x_i}$ corresponds to $x_i \otimes 1$. Hence $M \otimes_R k$ is $k$-spanned (and hence $R$-spanned) by the family $(x_i \otimes 1 \mid i \in I)$. But $[\varphi \otimes 1](e_i \otimes 1) = x_i \otimes 1$ and therefore each of the elements of this spanning family is in the image of $\varphi \otimes 1$ and surjectivty follows. By exactness we obtain that $C \otimes_R k = 0$, so Lemma \ref{lem: iso between tensoring with k and modding out by M plus} implies that $C / R_+ C = 0$, which gives that $R_+ C = C$, and by $\Gamma$-graded Nakayama's Lemma we obtain that $C = 0$, implying that $\varphi : F \to M$ is surjective. 

(b) $\implies$ (a). Now suppose that $(x_i \mid i \in I)$ generates $M$ as an $R$-module. Given any $\overline{x} \in M / R_+ M$ we can write $x = \sum_{i \in I} r_i x_i$ where all but finitely many of the $r_i$'s are zero. Then $\overline{x} = \sum_{i \in I} \overline{r_i x_i} = \sum_{i \in I} (r_i)_0 \overline{x_i}$ where the latter equality follows since if we write $r_i = \sum_{q \in Q} (r_i)_q$ with $(r_i)_q \in R_q$ then $\overline{r_i x_i} = \overline{ (r_i)_0 x_i}  + \overline{\sum_{q > 0} (r_i)_q x_i} = \overline{ (r_i)_0 x_i}$, this time the final equality holding because $\sum_{q > 0} (r_i)_q x_i \in R_+ M$. But $(r_i)_0 \in k$ so indeed $(\overline{x_i} \mid i \in I)$ $k$-spans $M / R_+ M$. 

\item Given $y \in \ker \varphi$, we can write $y = \sum_{j = 1}^n r_j e_{i_j}$ for $r_j \in R$. Applying $\varphi$ yields $0 = \sum_{j = 1}^n r_j x_{i_j}$ and taking residues modulo $R_+ M$ yields $0 = \sum_{j = 1}^n (r_j)_0 \overline{x}_{i_j}$. Since $(\overline{x}_i \mid i \in I)$ is $k$-independent, it follows that $(r_j)_0 = 0$ and hence that $r_j \in R_+$. Therefore $y \in R_+ F$.

\item (a) $\implies$ (b). Suppose that $(\overline{x_i} \mid i \in I)$ is a $k$-basis of $M / R_+ M$. Then by part (1) $(x_i \mid i \in I)$ generates $M$ as an $R$-module and if a subfamily $(x_j \mid j \in J)$ with $J \subseteq I$ $R$-generates $M$, then again by the same result, $(\overline{x_j} \mid j \in J)$ is a $k$-basis of $M / R_+ M$ which is only possible if $J = I$. 

(b) $\implies$ (a). Conversely, if $(x_i \mid i \in I)$ minimally generates $M$ as an $R$-module, then by the same corollary $(\overline{x_i} \mid i \in I)$ $k$-spans $M / R_+ M$. Therefore there is a subfamily $(\overline{x_j} \mid j \in J)$ which is a $k$-basis of $M / R_+ M$. Again by part (1), $(x_j \mid j \in J)$ $R$-generates $M$ so by minimality $I = J$ and hence $(\overline{x_i} \mid i \in I)$ is a $k$-basis of $M / R_+ M$.

(a) $\implies$ (c) follows from parts (1) and (2) of this proposition. 

(c) $\implies$ (a). That $(\overline{x}_i \mid i \in I)$ $k$-spans $M / R_+ M$ follows from part (1). To show it is $k$-independent, given $\gamma_i \in k$ for $i \in I$ with all but finitely many of the $\gamma_i$'s equal to zero such that $\sum_{i \in I} \gamma_i \overline{x}_i = 0$ in $M / R_+ M$, it follows that $\sum_{i \in I} \gamma_i x_i \in R_+ M$. Since $(x_i \mid i \in I)$ $R$-spans $M$, we can write $\sum_{i \in I} \gamma_i x_i = \sum_{i \in I} r_i x_i$ with $r_i \in R_+$ and all but finitely many of the $r_i$'s equal to zero. Then $\sum_{i \in I} \gamma_i e_i - \sum_{i \in I} r_i e_i \in \ker \varphi \subseteq R_+ F$, hence $\sum_{i \in I} \gamma_i e_i \in R_+ F$, so we can write $\sum_{i \in I} \gamma_i e_i = \sum_{i \in I} s_i e_i$ with $s_i \in R_+$. By independence of the $e_i$'s it follows that $\gamma_i \in R_+$, and since $R_+ \cap k = 0$, that $\gamma_i = 0$ for each $i \in I$. 

\end{enumerate}
\end{proof}

\begin{corollary}\label{cor: existence of a minimal R spanning set}
There exists a minimal $R$-spanning family for any BDF $\Gamma$-graded $R$-module $M$. 
\end{corollary}

\begin{proof}
Choose a homogeneous $k$-basis $(\overline{x_i} \mid i \in I)$ for $M /R_+ M$. Since $M / R_+ M$ is a BDF module, for each $g \in \Gamma$, $(M / R_+ M)_g$ is a finite-dimensional vector space, hence $\{ i\in I \mid \deg(x_i)  = g \}$ is finite. Also $\{ g \in \Gamma \mid \exists i \in I , \deg(x_i) = g \} = \Supp_\Gamma(M / R_+ M)$ which is finitely bounded below and downward finite, hence $(\overline{x_i} \mid i \in I)$ is a BDF family. By Proposition \ref{prop: k-gen for quotient gives r-gen for module}, $(x_i \mid i \in I)$ is a minimal $R$-spanning set for $M$. 
\end{proof}

Finally we are able to use the existence of a minimal $R$-spanning family to show that all projective BDF modules are free. 

\begin{theorem}\label{thm: projective implies free}
If $R$ is a PDCF $\Gamma$-graded $k$-algebra and $P$ is BDF $\Gamma$-graded $R$-module which is projective as an $R$-module then $P$ free and in fact isomorphic to $\bigoplus_{i \in I} R(-g_i)$ for some BDF family $(g_i \mid i \in I)$. 
\end{theorem}

\begin{proof}
By Proposition \ref{prop: k-gen for quotient gives r-gen for module} and Corollary \ref{cor: existence of a minimal R spanning set}, there exists a homogeneous minimal $R$-spanning family $( z_i \mid i \in I )$ of $P$ which is a BDF family and a BDF free $R$-module $F = \bigoplus_{i \in I} R(-g_i)$ where $g_i = \deg(z_i)$ and a graded surjective $R$-linear map $\varphi: F \to P$ with $K := \ker \varphi \subseteq R_+ F$. So we have a short exact sequence of BDF $\Gamma$-graded $R$-modules $0 \to K \to F \to P \to 0$ which splits because $P$ is projective. Therefore we obtain $F = K \oplus P'$ where $P'$ is a (not necessarily graded) submodule of $F$ isomorphic to $P$. Then we have that $K \subseteq R_+ F = R_+ ( K \oplus P') = R_+ K \oplus R_+P'$. Hence, given any $x \in K$, we can write $x = y + z$ where $y \in R_+ K$ and $z \in R_+ P'$. But subtraction yields $z = x - y \in K$, and since $z \in R_+ P' \subseteq P'$, and $P' \cap K = \{0\}$, it must be that $z = 0$ and $x = y \in R_+ K$. Hence $K \subseteq R_+K$. By $\Gamma$-graded Nakayama's Lemma we have that $K = 0$ and hence that $\varphi$ is an isomorphism, so $P$ is isomorphic to $F = \bigoplus_{i \in I} F(-g_i)$ as desired.
\end{proof}

One consequence of Theorem \ref{thm: projective implies free} is that ``$\Gamma$-graded projective'' is the same as ``$\Gamma$-graded and projective'' and ``$\Gamma$-graded free'' is the same as ``$\Gamma$-graded and free'', as the following corollary states precisely.

\begin{corollary}

\

\begin{enumerate}

\item If $P$ is a BDF $\Gamma$-graded $R$-module which is projective as an $R$-module then $P$ is projective in the category $R\text{-}\bMod_\Gamma^\text{BDF}$. 
 
\item If $F$ is a BDF $\Gamma$-graded $R$-module which is free as an $R$-module then $F \cong \bigoplus_{i \in I} R(-g_i)$ for some BDF family $(g_i \mid i \in I)$, i.e. $F$ has a homogeneous basis. 

\end{enumerate}
\end{corollary}


\subsection{Minimal Free Resolutions of BDF Modules}

We conclude this section with an application of the theory of minimal spanning sets we just developed to the existence of minimal free resolutions. These are free resolutions which map homogeneous bases onto minimal generating sets, and therefore, in some sense have the smallest possibly syzygy modules. They will be of importance to us when proving the BDF Hilbert Syzygy Theorem. 

\begin{lemma}\label{lem: minimal free resolution equiv cond}
Let $M$ be a BDF $\Gamma$-graded $R$-module and consider the following free resolution in the category $R\text{-}\bMod_\Gamma^\text{BDF}$. 
\begin{equation*}
\ldots \to F_2 \overset{\varphi_2}{\longrightarrow} F_1 \overset{\varphi_1}{\longrightarrow}  F_0 \overset{\varphi_0}{\longrightarrow} M \to 0
\end{equation*}
The following are equivalent:
\begin{enumerate}

\item $\im \varphi_n \subseteq R_+ F_{n - 1}$ for all $n \geq 1$. 

\item The maps in the complex 
\begin{equation*}
\ldots \to F_2 \otimes_R k \xrightarrow{\varphi_2 \otimes \id} F_1 \otimes_R k \xrightarrow{\varphi_1 \otimes \id}  F_0 \otimes_R k
\end{equation*}
are all $0$. 

\item For every $n \geq 0$, $\varphi_n$ maps every homogeneous basis of $F_n$ onto a minimal generating set for $\im \varphi_{n}$. 
 
 \end{enumerate}
 \end{lemma}
 
 \begin{proof}
(1) $\iff$ (2). First of all, for any any $n \geq 1$ the map $\varphi_n : F_n \to F_{n - 1}$ induces a map $\overline{\varphi}_n : F_n / R_+ F_n \to F_{n - 1} / R_+ F_{n - 1}$. Let $\psi_n : F_{n } \otimes_R k \to F_{n } / R_+ F_{n}$ denote the isomorphism described in Lemma \ref{lem: iso between tensoring with k and modding out by M plus}. Then the following diagram commutes. 

\begin{center}
\begin{tikzcd}
F_n \otimes_R k \arrow[r, "\varphi_n \otimes \id"] \arrow[d , "\psi_n" '] & F_{n - 1} \otimes_R k \arrow[d , "\psi_{n - 1}"] \\ 
F_n / R_+ F_n \arrow[r, "\overline{\varphi}_n" '] & F_{n -1} / R_+ F_{n - 1} 
\end{tikzcd}
\end{center}

Since $\psi_n$ and $\psi_{n - 1}$ are isomorphisms, the top map is zero if and only if the bottom map is zero. But the bottom map is zero if and only if $\im \varphi_n \subseteq R_+ F_{n - 1}$, as desired. 

 (1) $\iff$ (3). Given $n \geq 0$, let $(e_i \mid i \in I)$ denote a homogeneous basis for $F_n$ and $(x_i \mid i \in I)$ denote the image of this basis under $\varphi_n$. Then $(x_i \mid i \in I)$ is an $R$-spanning family for $\im \varphi_n$. By Proposition \ref{prop: k-gen for quotient gives r-gen for module}, this is a minimal $R$-spanning family for $\im \varphi_n$ if and only if $\ker \varphi_n \subseteq R_+ F$, but because the complex is exact, we have $\ker \varphi_n = \im \varphi_{n + 1}$, so the result follows. 
 \end{proof}
 
 \begin{definition}
A free resolution satisfying the three equivalent conditions of Lemma \ref{lem: minimal free resolution equiv cond} is called \emph{minimal}. 
\end{definition}

We now show that minimal free resolutions of BDF modules exist. 

\begin{proposition}\label{prop: existence of minimal gamma graded free resolutions}
Every BDF $\Gamma$-graded $R$-module $M$ has a minimal free resolution in the category $R\text{-}\bMod_\Gamma^\text{BDF}$. 
\end{proposition}

\begin{proof}
First of all, by Corollary \ref{cor: existence of a minimal R spanning set} there exists a homogeneous minimal $R$-generating set $( x_i \mid i \in I )$ of $M$. By Lemma \ref{lem: BDF surjection from a free module} there exists a BDF free $R$-module $F_0 = \bigoplus_{i \in I} F(-g_i)$ and a graded surjective $R$-linear map $\varphi_0 : F_0 \to M$ and by Proposition \ref{prop: k-gen for quotient gives r-gen for module} we have that $\ker \varphi_0 \subseteq R_+ F_0$. Replacing $M$ with $\ker \varphi_0$ we repeat this construction to recursively define an exact sequence 
\begin{equation*}
\ldots \to F_2 \overset{\varphi_2}{\longrightarrow} F_1 \overset{\varphi_1}{\longrightarrow}  F_0 \to M \to 0
\end{equation*}
with the property that $\im \varphi_n = \ker \varphi_{n - 1} \subseteq R_+ F_{n - 1}$ for all $n \geq 1$, hence by Lemma \ref{lem: minimal free resolution equiv cond} the resolution is minimal. 
\end{proof}


\section{$\Gamma$-finite Sequences of BDF Modules and the Koszul Complex}

In the following section we will prove the BDF Hilbert Syzygy Theorem which states that every BDF module has $\Gamma$-finite free resolution, i.e. a free resolution such that each graded piece is eventually zero. In this section we discuss the notion of $\Gamma$-finite sequence of BDF modules and we show that the Koszul complex associated to a BDF family of elements of $R$ is a $\Gamma$-finite chain complex which is exact if the family is a regular sequence. The Koszul complex will be used in the proof of the BDF Hilbert Syzygy Theorem as well. 

\begin{definition}
A sequence $\ldots \to M_2 \to M_1 \to M_0 \to 0$ of objects of $R\text{-}\bMod_\Gamma^\text{BDF}$ is called \emph{$\Gamma$-finite} if (GF1) $\bigcup_{n = 0}^\infty \Supp_\Gamma(M_n)$ is finitely bounded below and downward finite and (GF2) for all $g \in \Gamma$ the sequence of $k$-vector spaces $\ldots \to (M_2)_g \to (M_1)_g \to (M_0)_g \to 0$ is eventually $0$.
\end{definition}


\subsection{Exterior Powers of BDF Modules}

The Koszul complex requires graded exterior powers of BDF modules, which we now describe. 

\begin{proposition}
If $M$ is a BDF $\Gamma$-graded $R$-module, the submodule $L$ of $M^{\otimes n}$ spanned by elements of the form $x_1 \otimes \ldots \otimes x_n$ with $x_i = x_j$ for some $i \neq j$ is equal to the submodule $N$ spanned by elements of the form $x_1 \otimes \ldots \otimes x_n$ with $x_i$ homogeneous and $x_i = x_j$ for some $i \neq j$ together with elements of the form $x_1 \otimes \ldots \otimes x_j \otimes \ldots \otimes x_\ell \otimes \ldots \otimes x_n + x_1 \otimes \ldots \otimes x_\ell \otimes \ldots \otimes x_j \otimes \ldots \otimes x_n$ with $x_i$ homogeneous. 
\end{proposition}

\begin{proof}
See \cite{hazrat2016graded} Section 4.1. 
\end{proof}

\begin{corollary}\label{cor: exterior power is BDF}
Let $M$ be a BDF $\Gamma$-graded $R$-module. The quotient $\bigwedge^n M = M^{\otimes n} / L$ is a BDF $\Gamma$-graded $R$-module and for any $g \in \Gamma$, $\left(\bigwedge^n M \right)_g$ is $k$-spanned by elements of the form $x_1 \wedge \ldots \wedge x_n$ where $x_i \in M_{h_i}$ and $\sum_{i= 1}^n h_i = g$.
\end{corollary}

\begin{proof}
First of all, by Note \ref{not: tensor product of lots of BDF modules is BDF}, $M^{\otimes n}$ is a BDF $\Gamma$-graded $R$-module and $(M^{\otimes n})_g$ is $k$-spanned by elements of the form $x_1 \otimes \ldots \otimes x_n$ where $x_i \in M_{h_i}$ and $\sum_{i= 1}^n h_i = g$. Therefore, in the quotient the elements $x_1 \wedge \ldots \wedge x_k$ $k$-span $\left( \bigwedge^n M \right)_g$. 
\end{proof}

\begin{proposition}\label{prop: homog basis for exterior power}
Let $F$ be a BDF $\Gamma$-graded $R$-module with homogeneous $R$-basis $(e_i \mid i\in I)$. Fix any total order on $I$. Then $(e_{i_1} \wedge \ldots \wedge e_{i_n} \mid i_1 < \ldots< i_n )$ is a homogeneous $R$-basis for $\bigwedge^n F$ and $\deg(e_{i_1} \wedge \ldots \wedge e_{i_n}) = \deg(e_{i_1}) + \ldots + \deg(e_{i_n})$. Hence $\bigwedge^n F$ is also a free BDF $\Gamma$-graded $R$-module. 
\end{proposition}

\begin{proof}
That $(e_{i_1} \wedge \ldots \wedge e_{i_n} \mid i_1 < \ldots< i_n )$ is an $R$-basis for $\bigwedge^n F$ is well-known, and because $e_i$ is homogeneous for all $i \in I$, it follows that $e_{i_1} \otimes \ldots \otimes e_{i_n}$ is homogeneous in $F^{\otimes n}$ with $\deg(e_{i_1} \otimes \ldots \otimes e_{i_n}) = \deg(e_{i_1}) + \ldots + \deg(e_{i_n})$, and hence by definition of the quotient that $e_{i_1} \wedge \ldots \wedge e_{i_n}$ is homogeneous in $\bigwedge^n F$ and that $\deg(e_{i_1} \wedge \ldots \wedge e_{i_n}) = \deg(e_{i_1}) + \ldots + \deg(e_{i_n})$. 
\end{proof}


\subsection{The Koszul Complex}

\begin{definition}
The \emph{Koszul complex} associated to a BDF family $(r_n \mid n \in \nn)$ in a PDCF $\Gamma$-graded $k$-algebra $R$, denoted by $K(r_1, r_2, \ldots )$, is the following sequence of modules:
\begin{align*}
\ldots \to \bigwedge^3 F \overset{d_3}{\longrightarrow} \bigwedge^2 F \overset{d_2}{\longrightarrow} F \overset{d_1}{\longrightarrow} R \to 0
\end{align*}
Here $F = \bigoplus_{i \in \nn} R(-g_i)$ for $g_i = \deg(r_i)$ and we define $d_1(e_i) = r_i$ for all $i \in \nn$ and $d_n$ is the $R$-linear map defined on the homogeneous $R$-basis $( e_{i_1} \wedge \ldots \wedge e_{i_n} \mid i_1 < \ldots < i_n )$ for $\bigwedge^n F$ by
\begin{equation*}
e_{i_1} \wedge \ldots \wedge e_{i_n} \mapsto \sum_{\ell = 1}^n (-1)^\ell d_1(e_{i_\ell}) \cdot e_{i_1} \wedge \ldots \wedge \hat{e}_{i_\ell} \wedge \ldots \wedge e_{i_n}.
\end{equation*}
\end{definition}

\begin{proposition}
The Koszul complex associated to a BDF sequence $r_1, r_2, \ldots$ of elements of $R$ is a graded chain complex of BDF $\Gamma$-graded $R$-modules. 
\end{proposition}

\begin{proof}
By Lemma \ref{lem: BDF surjection from a free module}, $F$ is a BDF $\Gamma$-graded $R$-module and $d_1$ is a $\Gamma$-graded $R$-linear map. By Corollary \ref{cor: exterior power is BDF}, $\bigwedge^n F$ is also a BDF $\Gamma$-graded $R$-module and by Proposition \ref{prop: homog basis for exterior power}, $(e_{i_1} \wedge \ldots \wedge e_{i_n} \mid i_1 < \ldots < i_n )$ is a homogeneous basis for $\bigwedge^n F$, so $d_n$ is well-defined. This same Lemma also says that $e_{i_1} \wedge \ldots \wedge e_{i_n}$ has degree $g_{i_1} + \ldots + g_{i_n}$. Furthermore, $\deg(d_1(e_{i_\ell})) = \deg(r_{i_\ell}) = g_{i_\ell}$, hence $d_1(e_{i_\ell}) \cdot e_{i_1} \wedge \ldots \wedge \hat{e}_{i_\ell} \wedge \ldots \wedge e_{i_n}$ is homogeneous of degree $g_{i_1} + \ldots + g_{i_n}$, so $d_n(e_{i_1} \wedge \ldots \wedge e_{i_n})$ has degree $g_{i_1} + \ldots + g_{i_n}$ as well, and $d_n$ is graded. It is well-known that the Koszul complex is a chain complex, see e.g. \cite{eisenbud2013commutative} Section 17.4. 
\end{proof}

We now show that the Koszul complex of a BDF family is $\Gamma$-finite, preceded by the following technical lemma. 

\begin{lemma}\label{lem: T is finite}
If $r_1, r_2, \ldots$ is a BDF sequence of elements of $R$ with degrees $\deg(r_i) = g_i$, then for all $u \in \Gamma$, the set $\{ i \in \nn \mid g_i \leq_Q u\}$ is finite. 
\end{lemma}

\begin{proof}
Given $u \in \Gamma$, because $Q$ is downward finite, the set $S := \{ g \in Q \mid g \leq_Q u \}$ is finite. Then we notice that $T := \{ i \in \nn \mid g_i \leq_Q u \} = \bigcup_{g \in S} \{ i \in \nn \mid g_i = g \}$. Since $(r_i \mid i \in \nn)$ is a BDF family each set $\{ i \in \nn \mid g_i = g \}$ is finite, hence we have written $T$ as a finite union of finite sets. 
\end{proof}

\begin{proposition}
The Koszul complex associated to a BDF sequence $r_1, r_2, \ldots$ of elements of $R$ is $\Gamma$-finite. 
\end{proposition}

\begin{proof}
First of all, note that since $r_1, r_2, \ldots \in R$, we have that $g_1, g_2, \ldots \in Q$ and hence $\Supp_\Gamma(F) \subseteq Q$. Furthermore, recall that the grading on $F^{\otimes_R n}$ is given by $(F^{\otimes_R n})_g = \Span_k( x_1 \otimes \ldots \otimes x_n \mid x_i \in F_{h_i}, \sum_{i = 1}^n h_i = g )$. If $g \notin Q$, then at least one $h_i$ must be not in $Q$, but then $F_{h_i} = 0$, so $x_i = 0$, which implies $x_1 \otimes \ldots \otimes x_n = 0$. Hence $\Supp_\Gamma(F^{\otimes_R n}) \subseteq Q$. Since $\bigwedge^n F$ is a graded quotient of $F^{\otimes_R n}$, it follows that $\Supp_\Gamma(\bigwedge^n F) \subseteq Q$ as well. So we need only check the $\Gamma$-finite condition on elements of $Q$. 

Given $g \in \Gamma$, by Lemma \ref{lem: T is finite}, the set $T = \{ i \in \nn \mid g_i \leq_Q g\}$ is finite. Suppose that $m > |T|$. We claim that $\Supp_\Gamma(\bigwedge^m F) \cap (-\infty , g] = \emptyset$. Indeed, given any $h \leq_Q g$ with $h \in Q$ and any element $x \in (\bigwedge^m F)_h$, $x$ is a sum of elements of $\bigwedge^m F$ of the form $s \cdot e_{i_1} \wedge \ldots \wedge e_{i_m}$ for $s \in R$ homogeneous. But we have that $\deg(s) , g_{i_j} \in Q$, and $\deg(s) + \sum_{j = 1}^m g_{i_j} = h \leq_Q g$, hence for each $1 \leq j \leq m$ we have $i_j \in T$. Since $|T| = m - 1$, it must be that $i_n = i_\ell$ for some $n \neq \ell$. Therefore $e_{i_1} \wedge \ldots \wedge e_{i_m} = 0$, so $x = 0$. 

We have shown that for every $g  \in \Gamma$ there exits some $n \in \nn$ such that for all $m \geq n$, we have $(\bigwedge^m F)_h = 0$ for all $h \leq_Q g$. Furthermore $\bigcup_{n \in \nn} \Supp_\Gamma(\bigwedge^n F) \subseteq Q$, hence the Koszul complex is $\Gamma$-finite. 
\end{proof}

Recall that a family of elements $(r_i \mid i \in \nn)$ of a ring $R$ is said to be \emph{regular} on $R$ if the ideal generated by $r_1, r_2, \ldots$ is not equal to $R$ and for all $n \in \nn$, $r_{n + 1}$ is not a zero divisor on $R/ (r_1, \ldots, r_{n - 1})$. 

\begin{example}
The sequence $(x_{i , j})$ listed in any order in the polynomial ring $k[x_{i , j}]_n$ is clearly a regular sequence since quotienting by any number of the variables just yields a smaller polynomial ring which is a domain. 
\end{example}

\begin{proposition}\label{prop: Koszul complex of a regular sequence is a gamma finite free resolution}
Suppose $r_1, r_2, \ldots$ is a \textbf{regular} BDF sequence of elements of $R$. The Koszul complex is a $\Gamma$-finite BDF free resolution of $R / (r_1, r_2, \ldots )$, i.e. the following complex is exact:
\begin{align*}
\ldots \to \bigwedge^3 F \overset{d_3}{\longrightarrow} \bigwedge^2 F \overset{d_2}{\longrightarrow} F \overset{d_1}{\longrightarrow} R \to R / (r_1, r_2, \ldots ) \to 0
\end{align*}
\end{proposition}

\begin{proof}
First of all, the image of $d_1$ is $(r_1, r_2, \ldots )$ so the complex is exact at the $R$ and $R / (r_1, r_2, \ldots )$ terms. We must now show that $\ker d_n \subseteq \im d_{n + 1}$ for all $n \in \nn$. Since $d_n$ is a graded $R$-linear map, $\ker d_n$ is a graded submodule hence it suffices to show that $(\ker d_n)_u \subseteq \im d_{n + 1}$ for all $u \in \Gamma$. By Lemma \ref{lem: T is finite} the set $T = \{ i \in \nn \mid g_i \leq_Q u\}$ is finite. Then there exists $m \in \nn$ is such that for all $i \in T$, $i \leq m$. It is clear from the definition of a regular sequence that $r_1, \ldots, r_m$ is also regular on $R$. It is well-known that the Koszul complex $K(r_1, \ldots, r_m)$ is then a free resolution of $R / (r_1, \ldots, r_m)$ (again see \cite{eisenbud2013commutative} 17.4). Explicitly, let $F'= \bigoplus_{i = 1}^m R(-g_i)$ and for $n \in \nn$ define $d_n' : \bigwedge^n F' \to \bigwedge^{n - 1} F'$ as before. Then the following complex is exact:
\begin{align*}
\ldots \to \bigwedge^3 F' \overset{d_3'}{\longrightarrow} \bigwedge^2 F' \overset{d_2'}{\longrightarrow} F' \overset{d_1'}{\longrightarrow} R \to R / (r_1, \ldots, r_m ) \to 0
\end{align*}
We now define a morphism of complexes $\varphi: K(r_1, \ldots, r_m) \to K(r_1, r_2, \ldots)$, i.e. a collection of morphisms $\varphi_n$ such that the following diagram commutes. 

\begin{center}
\begin{tikzcd}
\ldots  \arrow[r] & \bigwedge^3 F'  \arrow[r, "d_3'"] \arrow[d, "\varphi_3"] & \bigwedge^2 F' \arrow[r, "d_2'"] \arrow[d, "\varphi_2"] & F' \arrow[r, "d_1'"] \arrow[d, "\varphi_1"] & R \arrow[r]\arrow[d, "\varphi_0 = \id"] & R / (r_1, \ldots, r_m ) \arrow[r] \arrow[d] & 0
\\ \ldots  \arrow[r] & \bigwedge^3 F  \arrow[r, "d_3"] & \bigwedge^2 F \arrow[r, "d_2"] & F \arrow[r, "d_1"] & R \arrow[r] & R / (r_1, r_2, \ldots ) \arrow[r] & 0
\end{tikzcd}
\end{center}

In fact we define $\varphi_1$ to just be the inclusion and define $\varphi_n = \bigwedge^n \varphi_1$. Then it's clear that the diagram commutes and that $\varphi_n$ is graded and injective because it sends the homogeneous basis elements of $\bigwedge^n F'$ to homogeneous basis elements of $\bigwedge^n F$ of the same degree. Any $x \in (\ker d_n)_u$ can be written as a sum of elements of the form $r \cdot e_{i_1} \wedge \ldots \wedge e_{i_n}$ which are of degree $u$. It follows that $g_{i_j} = \deg(e_{i_j}) \leq_Q u$ for all $1 \leq j \leq n$. Therefore $i_j \in T$ so $i_j \leq m$. Hence $e_{i_1} \wedge \ldots \wedge e_{i_n} \in \im \varphi_n$, hence $x \in \im \varphi_n$ as well, say $x = \varphi_n(y)$. Applying $d_n$ yields $0 = d_n(x) = d_n(\varphi_n(y)) = \varphi_{n - 1}(d'_n(y))$, but since $\varphi_{n - 1}$ is injective, we have that $d'_n(y) = 0$ so $y \in \ker d'_n$. By exactness of the top sequence there exists $z \in \bigwedge^{n + 1} F'$ such that $d_{n + 1}' (z) = y$. Therefore $x = \varphi_n(y) = \varphi_n(d_{n + 1}'(z)) = d_{n + 1} ( \varphi_{n + 1} (z))$, so $x \in \im \varphi_{n + 1}$ as desired. 
\end{proof}


\section{The BDF Hilbert Syzygy Theorem}

In this section we will prove the BDF $\Gamma$-graded analog of the Hilbert Syzygy Theorem, which says that every BDF module has a BDF $\Gamma$-finite free resolution. This will require a BDF $\Gamma$-graded $\Tor$ functor which we now describe.  


\subsection{The $\Gamma$-graded $\Tor$ Functor}

If $M$ is a BDF $\Gamma$-graded $R$-module, by Proposition \ref{prop: existence of minimal gamma graded free resolutions} there exists a BDF $\Gamma$-graded free resolution $F_\bullet$ of $M$ and tensoring $F_\bullet \to 0$ with another BDF module $N$ gives a chain complex $F_\bullet \otimes_R N \to 0$ whose terms are BDF $\Gamma$-graded $R$-modules by Proposition \ref{prop: tensor product grading} and whose maps are graded by Proposition \ref{prop: tensor product of graded maps is graded}. Therefore the kernels and images of these maps are also BDF $\Gamma$-graded $R$-modules, so the homology modules $\Tor^R_i(M, N) := H_i(F_\bullet \to 0)$ of this complex are BDF $\Gamma$-graded $R$-modules as well. It is well-known that if $F'_\bullet$ is another free resolution of $M$, then $H_i(F_\bullet \to 0)$ and $H_i(F'_\bullet \to 0)$ are canonically isomorphic as ungraded $R$-modules as are $\Tor^R_i(M, N)$ and $\Tor^R_i(N, M)$ (see \cite{weibel1994introduction} Section 2.7). However, we will need to know that these isomorphisms are graded. The former is standard, but we check the latter carefully since we could not find a reference.

\begin{proposition}\label{prop: symmetry of graded tor}
If $M$ and $N$ are BDF $\Gamma$-graded $R$-modules, there is a graded isomorphism $\Tor^R_i(M, N) \cong \Tor^R_i(N, M)$. 
\end{proposition}

\begin{proof}
We just follow the proof for the ungraded case in \cite{weibel1994introduction} and show that the isomorphism constructed there is $\Gamma$-graded. Let $\ldots F_1 \overset{\varphi_1}{\to} F_0 \overset{\varphi_0}{\to} M \to 0$ and $\ldots G_1 \overset{\psi_1}{\to} G_0 \overset{\psi_0}{\to} N \to 0$ be BDF $\Gamma$-graded free resolutions of $M$ and $N$ respectively. We form the double complex in Figure \ref{fig: tenor product of complexes}. 

\begin{figure}[h]
\begin{center}
\begin{tikzcd}
  & \vdots  \arrow[d] && \vdots  \arrow[d] && \vdots \arrow[d]  && \vdots \arrow[d] & 
\\ \ldots \arrow[r]  & F_2 \otimes_R G_2 \arrow[rr, "\varphi_2 \otimes \id_{G_2}"] \arrow[dd, "\id_{F_2} \otimes \psi_2"] && F_1 \otimes_R G_2 \arrow[rr, "\varphi_1 \otimes \id_{G_2}"] \arrow[dd, "\id_{F_1} \otimes \psi_2"] && F_0 \otimes_R G_2 \arrow[rr, "\varphi_0 \otimes \id_{G_2}"] \arrow[dd, "\id_{F_0} \otimes \psi_2"] && M \otimes_R G_2 \arrow[r] \arrow[dd, "\id_{M} \otimes \psi_2"] & 0
\\ & && && && &
\\ \ldots \arrow[r]  & F_2 \otimes_R G_1 \arrow[rr, "\varphi_2 \otimes \id_{G_1}"] \arrow[dd, "\id_{F_2} \otimes \psi_1"] && F_1 \otimes_R G_1 \arrow[rr, "\varphi_1 \otimes \id_{G_1}"] \arrow[dd, "\id_{F_1} \otimes \psi_1"] && F_0 \otimes_R G_1 \arrow[rr, "\varphi_0 \otimes \id_{G_1}"] \arrow[dd, "\id_{F_0} \otimes \psi_1"] && M \otimes_R G_1 \arrow[r] \arrow[dd, "\id_{M} \otimes \psi_1"] & 0
\\ & && && && &
\\ \ldots \arrow[r]  & F_2 \otimes_R G_0 \arrow[rr, "\varphi_2 \otimes \id_{G_0}"] \arrow[dd, "\id_{F_2} \otimes \psi_0"] && F_1 \otimes_R G_0 \arrow[rr, "\varphi_1 \otimes \id_{G_0}"] \arrow[dd, "\id_{F_1} \otimes \psi_0"] && F_0 \otimes_R G_0 \arrow[rr, "\varphi_0 \otimes \id_{G_0}"] \arrow[dd, "\id_{F_0} \otimes \psi_0"] && M \otimes_R G_0 \arrow[r] \arrow[dd] & 0
\\ & && && && &
\\ \ldots \arrow[r] & F_2 \otimes_R N \arrow[rr, "\varphi_2 \otimes \id_N"] \arrow[d] && F_1 \otimes_R N \arrow[rr, "\varphi_1 \otimes \id_N"] \arrow[d] && F_0 \otimes_R N \arrow[rr ] \arrow[d] && 0 \arrow[r] \arrow[d] & 0
\\ & 0 && 0 && 0 && 0 &
\end{tikzcd}
\end{center}
\caption{The Tensor Product of Two Free Resolutions}
\label{fig: tenor product of complexes}
\end{figure}

The homology of the bottom row is $\Tor^R_i(M, N)$ and that of the right-most column is $\Tor^R_i(N, M)$. Furthermore all rows except the bottom one are exact because $G_n$ is free, and hence flat, for all $n \in \nn$ and similarly all columns except the right-most are also exact. The proof in the ungraded case proceeds as follows. Given $\overline{x} \in \Tor^R_i(M, N)$ using exactness of the rows and columns as well as commutativity of the diagram, one constructs a ``zigzag'' sequence of elements $x = x_{i , -1} \in F_i \otimes N$, $x_{i , 0} \in F_i \otimes_R G_0$, $x_{i - 1, 0} \in F_{i - 1} \otimes_R G_0$, $x_{i - 1, 1} \in F_{i -1} \otimes_R G_1$, $x_{i - 2, 1} \in F_{i - 2} \otimes_R G_1$, \ldots $x_{0 , i} \in F_0 \otimes_R G_i$, $x_{-1 , i} \in M \otimes_R G_i$ such that if $j + \ell = i$, then $[\id_{F_j} \otimes \psi_\ell ](x_{j , \ell}) = x_{j , \ell - 1}$ and $[\varphi_j \otimes \id_{G_\ell} ](x_{j , \ell}) = x_{j - 1 , \ell}$. In other words the elements of the zigzag sequence map to each other under the maps in the zigzag. Then the isomorphism $\Tor^R_i(M, N) \to \Tor^R_i(N, M)$ is defined by $\overline{x} \mapsto \overline{x_{-1 , i}}$, and one checks that this is well-defined. 

In the $\Gamma$-graded case, we take $\overline{x} \in \Tor^R_i(M, N)_u = (\ker \varphi_i \otimes \id_N)_u / (\im \varphi_{i + 1} \otimes \id_N)_u$ so that $x \in (\ker \varphi_i \otimes \id_N)_u$. Then because all maps in this complex are graded (Proposition \ref{prop: tensor product of graded maps is graded}), Proposition \ref{prop: ker and im are graded} allows us to choose all terns of the zigzag sequence to be homogeneous of degree $u$. In particular $x_{-1 , i}$ is graded of degree $u$, hence $\overline{ x_{-1 , i} } \in \Tor^R_i(N, M)_u$ so indeed our isomorphism is graded. 
\end{proof}


\subsection{Existence of $\Gamma$-finite free resolutions}
 
\begin{lemma}\label{lem: tensor product of complexes is gamma finite} 
Suppose that $\ldots \to M_1 \to M_0 \to 0$ and $\ldots \to N_1 \to N_0 \to 0$ are $\Gamma$-graded sequences of BDF $\Gamma$-graded $R$-modules. Then the double complex $M_\bullet \otimes_R N_\bullet$ pictured in Figure \ref{fig: tenor product of sequences} is $\Gamma$-finite, meaning that (DGF1) the union of the $\Gamma$-supports is finitely bounded below and downward finite, and (DGF2) for any $g \in \Gamma$ there exists $n \in \nn$ such that $(M_i \otimes_R N_j)_g = 0$ for either $i \geq n$ or $j \geq n$. 
\end{lemma}

\begin{proof}
Let $U = \bigcup_{i = 1}^\infty \Supp_\Gamma(M_i)$ and $V = \bigcup_{j = 1}^\infty \Supp_\Gamma(N_j)$. Recall from Proposition \ref{prop: tensor product grading} that $(M_i \otimes_R N_j)_g = \Span_k( x \otimes y \mid \deg(x) + \deg(y) = g )$. Therefore $\Supp_\Gamma(M_i \otimes_R N_j) \subseteq \Supp_\Gamma(M_i) + \Supp_\Gamma(N_j) \subseteq U + V$. Since this is true for all $i , j$, it follows that $\bigcup_{i, j = 1}^\infty \left[ \Supp_\Gamma(M_i \otimes_R N_j) \right] \subseteq U + V$. But by definition of a $\Gamma$-finite sequence, $U$ and $V$ are finitely bounded below and downward finite, so by Lemma \ref{lem: sum of fbb df is too}, $U + V$ is as well. Therefore $\bigcup_{i, j = 1}^\infty \left[ \Supp_\Gamma(M_i \otimes_R N_j) \right]$ is contained in a finitely bounded below and downward finite subset of $\Gamma$, hence by Lemma \ref{lem: closure of fbb and df sets under taking subsets}, the double complex satisfies (DGF1). 

Given $g \in \Gamma$, because $U$ and $V$ are finitely bounded below and downward finite by hypothesis, by Lemma \ref{lem: finitely bounded below and downward finite sets have the finite sum property} the sets $W_1 = \{ u \in U \mid \exists v \in V , u + v = g\}$ and $W_2 = \{ v \in V \mid \exists u \in U , u + v = g\}$ are finite. Hence because $M_\bullet$ and $N_\bullet$ are $\Gamma$-finite, there exists some $n \in \nn$ such that if $i , j\geq n$ then $(M_i)_u = (N_j)_v = 0$ for all $u \in W_1$ and $v \in W_2$. But we have that $(M_i \otimes_R N_j)_g = \Span_k( x \otimes y \mid x \in (M_i)_u, y \in (M_j)_v, u + v = g )$ so supposing that either $i \geq n$ or $j \neq n$, if either $u \notin U$ or $v \notin V$ then $x = 0$ or $y = 0$ so $x \otimes y = 0$ and if $u \in U$ and $v \in V$ then $u + v = g$ so $u \in W_1$ and $v \in W_2$ hence either $x = 0$ or $y = 0$ so $x \otimes y = 0$ anyway. Hence if either $i \geq n$ or $j \geq n$ we have that $(M_i \otimes_R N_j)_g = 0$, so the double complex satisfies (DGF2). 
\end{proof}

\begin{figure}[h]
\begin{center}
\begin{tikzcd}
  & \vdots  \arrow[d] & \vdots  \arrow[d] & \vdots \arrow[d] 
\\ \ldots \arrow[r]  & M_2 \otimes_R N_2 \arrow[r] \arrow[d] & M_1 \otimes_R N_2 \arrow[r] \arrow[d] & M_0 \otimes_R N_2 \arrow[r] \arrow[d] & 0
\\ \ldots \arrow[r]  & M_2 \otimes_R N_1 \arrow[r] \arrow[d] & M_1 \otimes_R N_1 \arrow[r] \arrow[d] & M_0 \otimes_R N_1 \arrow[r] \arrow[d] & 0
\\ \ldots \arrow[r]  & M_2 \otimes_R N_0 \arrow[r] \arrow[d] & M_1 \otimes_R N_0 \arrow[r] \arrow[d] & M_0 \otimes_R N_0 \arrow[r] \arrow[d] & 0
\\ & 0 & 0 & 0 
\end{tikzcd}
\end{center}
\caption{The Tensor Product of two Sequences}
\label{fig: tenor product of sequences}
\end{figure}

\begin{corollary}\label{cor: tensoring a gamma finite resolution gives another gamma finite resolution} 
Let $\ldots \to M_1 \to M_0 \to 0$ be a $\Gamma$-finite chain complex consisting of BDF $\Gamma$-graded $R$-modules and let $N$ also be a BDF $\Gamma$-graded $R$-module. Then $\ldots \to M_1 \otimes_R N \to M_0 \otimes_R N \to 0$ is also $\Gamma$-finite. 
\end{corollary}

\begin{proof}
This follows by setting $N_0 = N$ and $N_i = 0$ for $i \geq 1$ in Lemma \ref{lem: tensor product of complexes is gamma finite}. 
\end{proof}

\begin{lemma}\label{lem: hilbert series of translated ring}
If $M$ is a BDF $\Gamma$-graded $R$-module and $g \in \Gamma$, then $\mathcal{H}_\Gamma( M(-g)) = \textbf{t}^g \mathcal{H}_\Gamma(M)$. 
\end{lemma}

\begin{proof}
We compute
\begin{align*}
\mathcal{H}_\Gamma(M(-g)) &= \sum_{h \in \Gamma} \dim( M(-g)_h ) \textbf{t}^h = \sum_{h \in \Gamma} \dim( M_{h - g} ) \textbf{t}^h \overset{(1)}{=} \sum_{h \in \Gamma} \dim( M_{h} ) \textbf{t}^{h+ g} \\&\overset{(2)}{=} \textbf{t}^g \sum_{h \in \Gamma} \dim( M_{h} ) \textbf{t}^{h}  = \textbf{t}^g \mathcal{H}_\Gamma( M). 
\end{align*}
Here (1) holds by the change of variables $h \mapsto h + g$ and (2) follows from the definition of the product in $\zz[[Q]]\{ \Gamma \}$.
\end{proof}

\begin{lemma}\label{lem: hilbert series of free tensor k}
If $(g_i \mid i \in I)$ is a BDF family of elements of $\Gamma$ and $F = \bigoplus_{i \in I} R(-g_i)$, then $F \otimes_R k \cong \bigoplus_{i \in I} k(-g_i)$. As a consequence, $\mathcal{H}_\Gamma(F \otimes_R k) = \sum_{g \in \Gamma} \# \{i \in I \mid g_i = g\} \textbf{t}^g$. 
\end{lemma}

\begin{proof}
We have that $F \otimes_R k \cong  \bigoplus_{i \in I} \left( R(-g_i) \otimes_R k \right) \cong \bigoplus_{i \in I} k(-g_i)$, where the first isomorphism follows because tensor product commutes with direct sum, and the second follows because for any $\Gamma$-graded $R$-module $M$ and any $g \in \Gamma$ we have $R(-g) \otimes_R M \cong M(-g)$ (see \cite{hazrat2016graded} Section 1.2.6 Equation 1.22). Since $\mathcal{H}_\Gamma(k) = 1$, it follows from Lemma \ref{lem: hilbert series of translated ring} that $\mathcal{H}_\Gamma(k(-g_i)) = \textbf{t}^{g_i}$. Therefore counting dimensions at each $g \in \Gamma$ yields $\mathcal{H}_\Gamma(F \otimes_R k) = \sum_{g \in \Gamma} \# \{i \in I \mid g_i = g\} \textbf{t}^{g}$.
\end{proof}

\begin{lemma}\label{lem: when is th gth part of a free module zero}
Let $(g_i \mid i \in I)$ be a BDF family of elements of $\Gamma$ and let $F = \bigoplus_{i \in I} R(-g_i)$ the associated BDF $\Gamma$-graded free module. Then $F_g = \bigoplus_{i \in I, g_i \leq_Q g } R(-g_i)_g$. In particular $F_g = 0$ if and only if $\{ i \in I \mid g_i \leq_Q g \}$ is empty. 
\end{lemma}

\begin{proof}
By definition $F_g = \bigoplus_{i \in I} R(-g_i)_g = \bigoplus_{i \in I} R_{g - g_i}$ and $R_{g - g_i} \neq 0$ if and only if $0 \leq_Q g - g_i$ which is equivalent to $g_i \leq_Q g$. Therefore the direct sum can be taken over the set $\{ i \in I \mid g_i \leq_Q g \}$, so we obtain the first desired result. If $\{ i \in I \mid g_i \leq_Q g \}$ is empty then obviously $F_g = 0$ since it is equal to a direct sum over the empty set. Conversely, if there exists $i \in I$ with $g_i \leq_Q g$ then $R(-g_i)_g \neq 0$ so $F_g = \bigoplus_{i \in I, g_i \leq_Q g } R(-g_i)_g \neq 0$.  
\end{proof}

\begin{lemma}\label{lem: free zero if and only if tensor with k is zero}
Let $F$ be a free BDF $\Gamma$-graded $R$-module and let $g \in \Gamma$. Then $F_g = 0$ if and only if $[F \otimes_R k]_{\leq g} = 0$ (where we define $M_{\leq g} := \bigoplus_{h \leq_Q g} M_h$ for any $\Gamma$-graded $R$-module $M$). 
\end{lemma}

\begin{proof}
By Lemma \ref{lem: when is th gth part of a free module zero} we have that $F_g = 0$ if and only if $\{ i \in I \mid g_i \leq_Q g \} = \emptyset$. But by Lemma \ref{lem: hilbert series of free tensor k}, this is true if and only if $\mathcal{H}_\Gamma(F \otimes_R k)_h = 0$ for all $h \leq_Q g$ which is true if and only if $[F \otimes_R k]_{\leq g} = 0$. 
\end{proof}

\begin{theorem}[BDF Hilbert Syzygy Theorem]\label{thm: BDF modules have gamma finite free resolutions}
Let $R$ be an PDCF $\Gamma$-graded $k$-algebra. Then for any BDF $\Gamma$-graded $R$-module $M$ with the property that $R_+$ is generated by a homogenous regular sequence, there exists a BDF $\Gamma$-finite free resolution of $M$. 
\end{theorem}

\begin{proof}
If $(r_i \mid i \in \nn)$ is a homogeneous regular sequence on $R$ which generates $R_+$ then by Proposition \ref{prop: Koszul complex of a regular sequence is a gamma finite free resolution}, the Koszul complex $K = K(r_1, r_2, \ldots)$ is a $\Gamma$-finite free resolution of $R / R_+ \cong k$. Tensoring the Koszul complex with $M$ we get a BDF $\Gamma$-graded complex $K_\bullet \otimes_R M$ which is $\Gamma$-finite by Corollary \ref{cor: tensoring a gamma finite resolution gives another gamma finite resolution}, and the $i$th homology of which is $\Tor^R_i(k , M)$. By Proposition \ref{prop: existence of minimal gamma graded free resolutions}, there exists a minimal BDF $\Gamma$-graded free resolution $F_\bullet$ of $M$. Tensoring with $k$ gives a complex $F_\bullet \otimes_R k$ the $i$th homology of which is $\Tor^R_i(M, k)$. But by Proposition \ref{prop: symmetry of graded tor}, we have that $\Tor^R_i(M, k)_g \cong \Tor^R_i(k, M)_g$ for all $g \in \Gamma$, hence the former is zero if and only if the latter is. Since $K_\bullet \otimes_R M$ is $\Gamma$-finite, for all $g \in \Gamma$ there exists $n \in \nn$ such that $[K_i \otimes_R M]_g = 0$, and hence $\Tor^R_i(k , M) = 0$, and hence $\Tor^R_i(M, k)_g = 0$ for all $i \geq n$. 

Because the resolution $F_\bullet$ of $M$ is minimal, the differentials of the complex $F_\bullet \otimes_R k$ are all $0$ by Lemma \ref{lem: minimal free resolution equiv cond}, and hence $\Tor^R_i(M, k)_g = (F_i \otimes_R k)_g$. But by Lemma \ref{lem: free zero if and only if tensor with k is zero}, $(F_i \otimes_R k)_g = 0$ if and only if $(F_i)_g = 0$, therefore $F_\bullet$ is a $\Gamma$-finite free resolution of $M$ as desired. 
\end{proof}


\section{Projective $\Gamma$-Dimension of a BDF Module}

Although we will not need it in the sequel, we give an application of the previous section to defining the graded projective dimension of a BDF module and we show that it is equal to a certain graded torsion dimension, just as in the ungraded finitely generated case (see \cite{eisenbud2013commutative} Corollary 19.5). 

Let $M$ be a BDF $\Gamma$-graded $R$-module. The \emph{projective $\Gamma$-dimension} of $M$ is the function $\text{pd}_R^\Gamma(M) : \Gamma \to \zz_{\geq 0} \cup \{ \infty \}$ which sends $g \in \Gamma$ to the minimum $n$ such that there exists a BDF projective resolution $P_\bullet$ of $M$ with $(P_n)_g \neq 0$ but $(P_{i})_g = 0$ for all $i > n$. The \emph{torsion $\Gamma$-dimension} of $M$ is the function $\text{td}_R^\Gamma(M) : \Gamma \to \zz_{\geq 0} \cup \{ \infty \}$ which sends $g \in \Gamma$ to the minimum $n$ such that $\Tor^R_n(k , M)_{\leq g} \neq 0$ but $\Tor^R_i(k , M)_{\leq g} = 0$ for all $i > n$. In \cite{miller2005combinatorial} the $i$th graded Betti number of $M$ in degree $g$ is defined to be $\beta_{i , g}(M) = \dim_k \Tor^R_i(k , M)_{g}$, so in this language the torsion dimension $\text{td}_R^\Gamma(M)(g)$ is just the first $n$ such that all Betti numbers $\beta_{n , h}(M)$ for $h \leq_Q g$ vanish. 

\begin{proposition}
If $M$ is a BDF $\Gamma$-graded $R$-module, then for each $g \in \Gamma$, we have $\text{pd}_R^\Gamma(M)(g) = \text{td}_R^\Gamma(M)(g)$. 
\end{proposition}

\begin{proof}
First we show that $\text{pd}_R^\Gamma(M)(g) \geq \text{td}_R^\Gamma(M)(g)$. This is trivially true if $\text{pd}_R^\Gamma(M)(g) = \infty$, so suppose that $\text{pd}_R^\Gamma(M)(g) = n \in \zz_{\geq 0}$. Then there exits a BDF free resolution $F_\bullet$ of $M$ such that $(F_i)_g = 0$ for all $i > n$. Tensoring with $k$ yields a complex $F_\bullet \otimes_R k$ which, by Lemma \ref{lem: free zero if and only if tensor with k is zero}, has the property that $(F_i \otimes_R k)_{\leq g} = 0$ for all $i > n$ so $\Tor^R_i(k , M)_{\leq g} = 0$ for all $i > n$. 

Next we show that $\text{pd}_R^\Gamma(M)(g) \leq \text{td}_R^\Gamma(M)(g)$. Again this is trivial if $\text{td}_R^\Gamma(M)(g) = \infty$ so suppose that $\text{td}_R^\Gamma(M)(g) = n \in \zz_{\geq 0}$. By Proposition \ref{prop: existence of minimal gamma graded free resolutions}, there exists a minimal BDF $\Gamma$-graded free resolution $F_\bullet$ of $M$. Tensoring with $k$ gives a complex $F_\bullet \otimes_R k$ whose differentials are all $0$. It follows that $\Tor^R_i(k , M)_g = (F_i \otimes_R k)_g$. Note that by Lemma \ref{lem: free zero if and only if tensor with k is zero}, $(F_i \otimes_R k)_{\leq g} = 0$ if and only if $(F_i)_g = 0$, and for $i > n$ we have that the former is $0$ so the latter is too, implying that $\text{pd}_R^\Gamma(M)(g) \leq n$. 
\end{proof}


\section{$K$-series of Graded Modules}

The rest of the paper will be dedicated to defining the Grothendieck group of our abelian category $R\text{-}\textbf{Mod}_{\Gamma}^{\text{BDF}}$ and proving it is isomorphic to $\zz[[Q]\{\Gamma\}$. This isomorphism is obtained by computing the $K$-series of a BDF module, which is a slightly altered version of the Hilbert series. In this section we will introduce $K$-series and, by turning the set of such into a linear topological module, show that a free BDF module is determined by both its Hilbert series and its $K$-series


\subsection{Definition of $K$-series} 

\begin{lemma}
Let $R$ be a PDCF $\Gamma$-graded $k$-algebra. Then $\mathcal{H}_\Gamma(R) \in \zz[[Q]]\{\Gamma\}$ is a unit and $\mathcal{H}_\Gamma(R)^{-1} \in \zz[[Q]]$.
\end{lemma}

\begin{proof}
Let $\mathcal{H}_\Gamma(R) = a = \sum_{g \in Q} a_g \textbf{t}^g$, and note that we may take the sum over $Q$ instead of $\Gamma$ by definition of $Q$. Now because $R_0 = k$, it follows that $a_0 =  \dim_k(R_0) = 1$. Since $Q$ is well-founded by hypothesis, we can use well-founded recursion to define an element $b =  \sum_{g \in Q} a_g \textbf{t}^g \in \zz[[Q]]\{\Gamma\}$ which will be the inverse of $a$. Let $b_0 = 1$. Then suppose that for some $g \in Q$ we have defined $b_h$ for all $h \in Q$ with $h <_Q g$. Then we let 
\begin{equation*}
b_g = - \sum_{\substack{u , v \in Q \smallsetminus \{ 0 \} \\ u + v = g}} b_u a_v,
\end{equation*}
which makes sense because if $u, v \in Q \smallsetminus \{ 0 \}$ and $u + v = g$ then $u <_Q g$ hence $b_u$ is already defined by induction. 

Having defined $b$, we let $c = a b$ and we note first that $c_0 = \sum_{\substack{u , v \in Q \\ u + v = 0}} a_u b_v= a_0 b_0 = 1$ because $Q$ is pointed by hypothesis. For $h \in Q$ with $h \neq 0$, we have 
\begin{equation*}
c_h = \sum_{\substack{u , v \in Q \\ u + v = h}} a_u b_v = b_h a_0 + \sum_{\substack{u , v \in Q \smallsetminus \{ 0 \} \\ u + v = h}} b_u a_v = 0. 
\end{equation*}
Therefore $c = 1$ as desired. From the definition of $b$ it is clear that $\mathcal{H}_\Gamma(R)^{-1} \in \zz[[Q]]$ ($\zz[[Q]]$ denotes the Laurent series which are supported on $Q$). 
\end{proof}

\begin{definition}
The \emph{$K$-series} of a $\Gamma$-graded $R$-module $M$ is defined to be $\mathcal{K}_\Gamma(M) = \mathcal{H}_\Gamma(M) \cdot \mathcal{H}_\Gamma(R)^{-1}$. 
\end{definition}

An important property of $K$-series is that they are additive, as described in the following result. 

\begin{proposition}\label{prop: additivity of Hilbert series}
Hilbert series and $K$-series are additive i.e. if $0 \to L \to M \to N \to 0$ is an exact sequence of BDF $R$-modules then $\mathcal{H}_\Gamma(L) - \mathcal{H}_\Gamma(M) + \mathcal{H}_\Gamma(N) = 0$ and $\mathcal{K}_\Gamma(L) - \mathcal{K}_\Gamma(M) + \mathcal{K}_\Gamma(N) = 0$. 
\end{proposition}

\begin{proof}
If $0 \to L \to M \to N \to 0$ is an exact sequence of BDF $R$-modules then because kernels and images of graded module homomorphisms are graded submodules, it follows that $0 \to L_g \to M_g \to N_g \to 0$ is an exact sequence of $k$-vector spaces for all $g \in \Gamma$. This implies that $\dim_k(L_g) - \dim_k(M_g) + \dim_k(N_g) = 0$, which is exactly the coefficient of $\textbf{t}^g$ in $\mathcal{H}_\Gamma(L) - \mathcal{H}_\Gamma(M) + \mathcal{H}_\Gamma(N)$. Hence the latter is $0$. Multiplying by $\mathcal{H}_\Gamma(R)^{-1}$ gives the final desired result. 
\end{proof}


\subsection{Linear Topological Modules}

Here we give a brief summary of the basics of linear topological modules that we will need to define a topology on $\zz[[Q]]\{\Gamma\}$. A \emph{directed filtration} of an $R$-module $M$ consists of a non-empty directed set $D$ (i.e. a set with a reflexive and transitive binary operation where every pair of elements has an upper bound) together with an order-reversing function $F : D \to \Sub_R(M)$ (where $\Sub_R(M)$ denotes the set of $R$-submodules of $M$ and is a poset under inclusion). By definition, if $d , e \in D$ then there exists $t \in D$ with $d \leq t$ and $e \leq t$, and hence $F_d \supseteq F_t$ and $F_e \supseteq F_t$, implying that $F_t \subseteq F_d \cap F_e$. A \emph{filtered $R$-module} is an $R$-module with a directed filtration on it.

Recall that a \emph{topological $R$-module} is an $R$-module $M$ with a topology on it such that addition $M \times M \to M$ which sends $m , n \mapsto m + n$, and scalar multiplication $R \times M \to M$ which sends $r , m \mapsto r m$ is continuous (where $R$ is given the discrete topology). 

\begin{proposition}\label{prop: filtered modules give topological modules}
Let $(F , D)$ be a directed filtration of an $R$-module $M$. 

\begin{enumerate}

\item The set $B = \{ a + F_d \mid a \in \Gamma, d \in D \}$ is a base for a topology on $M$ which we call the \emph{$F$-adic} topology. 

\item The $F$-adic topology makes $M$ into a topological $R$-module. 

\end{enumerate}
\end{proposition}

\begin{proof}
\begin{enumerate}

\item Clearly $B$ covers $M$ since $D$ is non-empty and for any $x \in M$ and any $d \in D$, we have that $x \in x + F_d$. Next, given $x + F_d$ and $y + F_e$ with $x , y \in M$ and $d , e \in D$, if $z \in (x + F_d) \cap (y + F_e)$ then by hypothesis there exists $t \in D$ such that $F_t \subseteq F_d \cap F_e$. Clearly then we have $z + F_t \subseteq z + F_d = x + F_d$ and similarly $z + F_t \subseteq y + F_e$, hence $z + F_t \subseteq (x + F_d) \cap (y + F_e)$ as desired. 

\item We first show that addition $M \times M \to M$ is continuous. Suppose that $x + y \in z + F_d$. Then $x + F_d$ and $y + F_d$ are neighborhoods of $x$ and $y$ respectively and $(x + F_d) + (y + F_d) = x + y + F_d = z + F_d$, so $(x + F_d) \times (y + F_d)$ is a neighborhood of $(x , y)$ in $M \times M$ which maps into $z + F_d$ under $+$, hence indeed addition is continuous. Finally to show that scalar multiplication is continuous, suppose $rx \in z + F_d$. Then we have that $r (x + F_d) = rx + F_d = z + F_d$. Hence $\{ r \} \times (x + F_d)$ is a neighborhood of $(r , x)$ which maps into $z + F_d$ under scalar multiplication as desired.

\end{enumerate}
\end{proof}

\begin{definition}
A topological $R$-module $M$ is called a \emph{linear topological $R$-module} if there exists a directed filtration of $M$ whose associated topology agrees with that on $M$. 
\end{definition}

We now discuss convergence of series in a linear topological module. Let $(F , D)$ be a directed filtration of an $R$-module $M$. A net $(x_i \mid i \in I)$ in $M$ (i.e. a family of elements of $M$ indexed by a directed set $I$) converges to $x \in M$ if and only if for all $d \in D$ there exists $i \in I$ such that if $j \geq i$, then $x_j - x \in F_d$. If $(x_a \mid a \in A)$ is a family of elements of $M$, by definition the series $\sum_{a \in A}^\infty x_a$ converges to $x$ if its net of partial sums $( \sum_{b \in B} x_b \mid B \in \text{Fin}(A) )$ does, where $\text{Fin}(A)$ denotes the set of finite subsets of $A$ which is directed by containment. Explicitly this means for all $d \in D$ there exists a finite set $B \subseteq A$ such that if $B \subseteq C \subseteq A$ and $C$ is finite then $x- \sum_{c \in C} x_c \in F_d$. 

Later we will need to know that series in our linear topological modules converge to at most one point. The following result gives us a method of checking when this is true. 

\begin{proposition}\label{prop: hausdorff if and only if all nets converge}
Let $(F , D)$ be a directed filtration of an $R$-module $M$. Then the $F$-adic topology on $M$ is Hausdorff if and only if every net $(x_i \mid i \in I)$ in $M$ converges to at most one point. 
\end{proposition}

\begin{proof}
This is true for general topological spaces, see \cite{kelley1975general} Chapter 2, Theorem 3. 
\end{proof}

Furthermore, for linear topological modules, the Hausdorff condition can be checked easily simply by taking the intersection of the defining filtration. 

\begin{proposition}\label{prop: hausdorff if and only if intersection of subspaces is zero}
If $(F , D)$ is a directed filtration of an $R$-module $M$ then the $F$-adic topology on $M$ is Hausdorff if and only if $\bigcap_{d \in D} F_d = 0$. 
\end{proposition}

\begin{proof}
See \cite{matsumura1989commutative} Section 8. 
\end{proof}


\subsection{The Topology of $\zz[[Q]]\{\Gamma\}$}

 We now apply the theory of linear topological modules that we just developed to our study of $K$-series. First we note that there is an injective ring homomorphism $\zz[[Q]] \to \zz[[Q]] \{ \Gamma \}$ which makes $\zz[[Q]] \{ \Gamma \}$ into a $\zz[[Q]]$ module. Now, given any $g \in \Gamma$, define $Z_g$ to be the set of series $a \in \zz[[Q]]\{\Gamma\}$ with $\Supp_\Gamma(a) \cap (-\infty, g] = \emptyset$ where $(-\infty , g] = \{ h \in \Gamma \mid h \leq_Q g \}$. 

\begin{proposition}
$Z_g$ is a $\zz[[Q]]$ submodule of $\zz[[Q]]\{\Gamma\}$. 
\end{proposition}

\begin{proof}
It is obvious that $Z_g$ is a subgroup of $\zz[[Q]]\{\Gamma\}$. Given $a \in Z_g$ and $b \in \zz[[Q]]$, for any $u \in \Gamma$ with $u \leq_Q g$ by definition we have that $(ab)_u = \sum_{g + h = u} a_g b_h$. If $h \notin Q$ then $b_h = 0$ so $a_g b_h = 0$. But if $h \in Q$ then $g + h = u$ implies that $g \leq_Q u \leq_Q g$ and hence $a_g = 0$. Therefore $(ab)_u = 0$ and $ab \in Z_g$.
\end{proof}

Given any subset $U$ of $\Gamma$, define $(-\infty , U] = \{ g \in \Gamma \mid \exists u \in U, g \leq_Q u \}$. Then let $\mathcal{P}_\text{fin}(\Gamma)$ be the set of finite subsets $U \subseteq \Gamma$ with the relation $U \leq U'$ if and only if for all $u \in U$ there exists $u' \in U'$ such that $u \leq_Q u'$. For each $U \subseteq \Gamma$, define $Z_U$ to be the subset consisting of series $f \in \zz[[Q]]\{\Gamma\}$ with $\Supp_\Gamma(f) \cap (-\infty , U] = \emptyset$. Note that for any subset $U \subseteq \Gamma$ we have that $Z_{U} = \bigcap_{u \in U} Z_u$, hence $Z_U$ is a $\zz[[Q]]$-submodule of $\zz[[Q]]\{\Gamma\}$ for all $U \in \mathcal{P}_\text{fin}(\Gamma)$. 

\begin{proposition}
The set $\mathcal{P}_\text{fin}(\Gamma)$ is directed and the function $\mathcal{P}_\text{fin}(\Gamma) \to \Sub_{\zz[[Q]]}(\zz[[Q]]\{\Gamma\})$ sending $U \mapsto Z_U$ is a directed filtration.
\end{proposition}

\begin{proof}
The relation is clearly reflexive and transitive, and for any such $U$ and $U'$, $U \cup U'$ is an upper bound for $U$ and $U'$ which is also finite. If $U \leq U'$, then clearly $(\infty , U] \subseteq (-\infty , U']$, so for any $f \in Z_{U'}$ we have that $\Supp_\Gamma(f) \cap (-\infty , U'] = \emptyset$ so certainly $\Supp_\Gamma(f) \cap (-\infty , U] = \emptyset$ and hence $f \in Z_U$. Hence the map is order-reversing. 
\end{proof}

\begin{proposition}\label{prop: z adic topology is hausdorff}
We have that $\bigcap_{U \in \mathcal{P}_\text{fin}(\Gamma)} Z_U = 0$. As a result the $Z$-adic topology on $\zz[[Q]]\{\Gamma\}$ is Hausdorff and nets in $\zz[[Q]]\{\Gamma\}$ converge to at most one point. 
\end{proposition}

\begin{proof}
Suppose that $f \in \bigcap_{U \in \mathcal{P}_\text{fin}(\Gamma)} Z_U$. Then given any $g \in \Gamma$, we have that $f \in Z_g$ so in particular $f_g = 0$. Since $g$ was arbitrary, it follows that $f = 0$. The desired conclusions follow from Propositions \ref{prop: hausdorff if and only if all nets converge} and \ref{prop: hausdorff if and only if intersection of subspaces is zero}. 
\end{proof}

Our main application of this topology is to compute the $K$-series of a free BDF module. We begin with two technical lemmas regarding the multiplication of convergent series by elements of the ring, and then prove the desired resut over the course of the next three lemmas. 

\begin{lemma}\label{lem: multiplication by r is continuous}
If $(F , D)$ is a directed filtration of an $R$-module $M$ then for any $r \in R$, multiplication by $r$ gives a continuous function $M \to M$. 
\end{lemma}

\begin{proof}
Multiplication by $r$ is equal to the composition of the continuous functions $M \to R \times M$ given by $m \mapsto (r , m)$ and scalar multiplication $R \times M \to M$. 
\end{proof}

\begin{lemma}\label{lem: multiply a convergent series by an element}
If $(F , D)$ is a directed filtration of an $R$-module $M$ and a series $\sum_{a \in A}^\infty x_a$ in $M$ converges to $x$ then the series $\sum_{a \in A}^\infty r x_a$ converges to $rx$. 
\end{lemma}

\begin{proof}
By Lemma \ref{lem: multiplication by r is continuous}, multiplication by $r$ is a continuous function $M \to M$, and continuous functions preserve convergence of nets, hence the net $( r \sum_{b \in B} x_b \mid B \in \text{Fin}(A) )$ converges to $rx$, but because each $B$ is a finite set, we obtain $r \sum_{b \in B} x_b = \sum_{b \in B} r x_b$, so this is actually the net of partial sums of the series $\sum_{a \in A}^\infty r x_a$. 
\end{proof}

\begin{lemma}\label{lem: k series distribute over free direct sums}
If $(g_i \mid i \in I)$ is a BDF family of elements of $\Gamma$ then $\mathcal{K}_\Gamma( \bigoplus_{i \in I} R(-g_i)) = \sum_{i \in I} \mathcal{K}_\Gamma(R(-g_i))$, i.e. the RHS converges to the LHS in the $Z$-adic topology. 
\end{lemma}

\begin{proof}
Let $F  = \bigoplus_{i \in I} R(-g_i)$. We first show that $\mathcal{H}_\Gamma( F ) = \sum_{i \in I} \mathcal{H}_\Gamma(R(-g_i))$. Let $U \subseteq \Gamma$ be a finite set. Because $(g_i \mid i \in I)$ is a BDF family, the set $\{g_i \mid i \in I\}$ is downward finite, hence for each $u \in U$ the set $\{ g_i \mid g_i \leq_Q u \}$ is finite and furthermore for each $g_i$ in this set, the set $\{ j \in I \mid g_j = g_i \}$ is also finite. Since there are only finitely many elements of $U$, the set $B = \{ i \in I \mid \exists g_i \in (-\infty , U] \}$ is finite. If $C \subseteq I$ is finite and $B \subseteq C$, we claim that $\mathcal{H}_\Gamma(F)  - \sum_{i \in C} \mathcal{H}_\Gamma(R(-g_i)) \in Z_U$. First note that because $C$ is finite, we have that $\sum_{i \in C} \mathcal{H}_\Gamma(R(-g_i)) = \mathcal{H}_\Gamma( \bigoplus_{i \in C} R(-g_i) )$. 

Then if $g \in (-\infty , U]$, by Lemma \ref{lem: when is th gth part of a free module zero}, $\left( \bigoplus_{i \in I} R(-g_i)\right)_g =  \bigoplus_{i \in I, g_i \leq_Q g } R(-g_i)_g$, and similarly $\left(\bigoplus_{i \in C} R(-g_i)\right)_g = \bigoplus_{i \in C, g_i \leq_Q g } R(-g_i)_g$. But if $g_i \leq_Q g$ then $i \in B \subseteq C$, hence these two direct sums are taken over the same set, i.e. $\left(\bigoplus_{i \in I} R(-g_i)\right)_g = \left(\bigoplus_{i \in C} R(-g_i)\right)_g$. In particular their dimensions are the same. Since $g$ was an arbitrary element of $(-\infty , U]$, the Hilbert series $\mathcal{H}_\Gamma(F)$ and $\mathcal{H}_\Gamma( \bigoplus_{i \in C} R(-g_i) )$ agree on $(-\infty , U]$, so indeed their difference is in $Z_U$. It follows that $\mathcal{H}_\Gamma( F ) = \sum_{i \in I} \mathcal{H}_\Gamma(R(-g_i))$ as claimed. 

Finally, using the fact that $\mathcal{H}_\Gamma(R)^{-1} \in \zz[[Q]]$ and Lemma \ref{lem: multiply a convergent series by an element} in step (1), we can compute the desired result:
\begin{align*}
\mathcal{K}_\Gamma \left( \bigoplus_{i \in I} R(-g_i) \right) &= \mathcal{H}_\Gamma(R)^{-1} \mathcal{H}_\Gamma \left( \bigoplus_{i \in I} R(-g_i) \right)  = \mathcal{H}_\Gamma(R)^{-1} \sum_{i \in I} \mathcal{H}_\Gamma(R(-g_i)) 
\\&\overset{(1)}{=} \sum_{i \in I} \mathcal{H}_\Gamma(R)^{-1} \mathcal{H}_\Gamma(R(-g_i)) =  \sum_{i \in I} \mathcal{K}_\Gamma(R(-g_i)). 
\end{align*}
\end{proof}

\begin{lemma}\label{lem: k series of translated ring}
Let $M$ be a BDF $\Gamma$-graded $R$-module, and $g \in \Gamma$. Then $\mathcal{K}_\Gamma(M(-g)) = \textbf{t}^g \mathcal{K}_\Gamma(M)$ and in particular we have $\mathcal{K}_\Gamma( R(-g) ) = \textbf{t}^{g}$. 
\end{lemma}

\begin{proof}
By Lemma \ref{lem: hilbert series of translated ring} we have that $\mathcal{H}_\Gamma(M(-g)) = \textbf{t}^g \mathcal{H}_\Gamma( M)$, so we compute: 
\begin{equation*}
\mathcal{K}_\Gamma(M(-g)) = \mathcal{H}_\Gamma(M(-g)) \mathcal{H}_\Gamma(R)^{-1} = \textbf{t}^g \mathcal{H}_\Gamma( M)\mathcal{H}_\Gamma(R)^{-1} = \textbf{t}^g \mathcal{K}_\Gamma(M) 
\end{equation*}
In particular this implies that $\mathcal{K}_\Gamma( R(-g) ) = \textbf{t}^{g} \mathcal{K}_\Gamma( R ) = \textbf{t}^{g} \mathcal{H}_\Gamma( R ) \mathcal{H}_\Gamma( R )^{-1}  = \textbf{t}^{g}$. 
\end{proof}

\begin{lemma}\label{lem: sum of linear terms}
If $(g_i \mid i \in I)$ is a BDF family of elements of $\Gamma$ then $\sum_{i \in I} \textbf{t}^{g_i} = \sum_{g \in \Gamma} \# \{ i \in I \mid g_i = g \} \textbf{t}^g$. 
\end{lemma}

\begin{proof}
Let $f_g = \# \{ i \in I \mid g_i = g \}$ so that $f = \sum_{g \in \Gamma} \# \{ i \in I \mid g_i = g \} \textbf{t}^g$. Given any finite set $U \subseteq \Gamma$, because $(g_i \mid i \in I)$ is a BDF family, the set $B = \{ i \in I \mid g_i \in (-\infty , U] \}$ is finite. If $C \subseteq \Gamma$ is finite and $B \subseteq C$ then we have that $f - \sum_{i \in C} \textbf{t}^{g_i} \in Z_{U}$ because for any $g \in (-\infty, U]$, we have that $f_g$ is the size of the set $\{ i \in I \mid g_i = g\}$, but this set is contained in $B$ and hence in $C$, so there are exactly $f_g$ terms of the sum $\sum_{i \in C} \textbf{t}^{g_i}$ equal to $\textbf{t}^g$.
\end{proof}

We are now able to compute the $K$-series of a free BDF module. 

\begin{proposition}\label{prop: k series of a free module}
If $(g_i \mid i \in I)$ is a BDF family of elements of $\Gamma$ then $\mathcal{K}_\Gamma( \bigoplus_{i \in I} R(-g_i)) = \sum_{g \in \Gamma} \# \{ i \in I \mid g_i = g \} \textbf{t}^g$.
\end{proposition}

\begin{proof}
We compute: 
\begin{equation*}
\mathcal{K}_\Gamma( \bigoplus_{i \in I} R(-g_i)) \overset{(1)}{=} \sum_{i \in I} \mathcal{K}_\Gamma(R(-g_i)) \overset{(2)}{=} \sum_{i \in I} \textbf{t}^{g_i}  \overset{(3)}{=} \sum_{g \in \Gamma} \# \{ i \in I \mid g_i = g \} \textbf{t}^g
\end{equation*}
Here (1) follows from Lemma \ref{lem: k series distribute over free direct sums}, (2) follows from Lemma \ref{lem: k series of translated ring}, and (3) follows from Lemma \ref{lem: sum of linear terms}. 
\end{proof} 

As a consequence, we obtain that free BDF modules are determined by their $K$-series. 

\begin{corollary}\label{cor: same hilbert series implies isomorphic}
If $F$ and $F'$ are free BDF $R$-modules then the following are equivalent: 
\begin{enumerate}

\item $\mathcal{H}_\Gamma(F) = \mathcal{H}_\Gamma(F')$

\item $\mathcal{K}_\Gamma(F) = \mathcal{K}_\Gamma(F')$

\item $F \cong F'$

\end{enumerate}  
\end{corollary}

\begin{proof}
One obtains (1) $\iff$ (2) by multiplying by $\mathcal{H}_\Gamma(R)$ or $\mathcal{H}_\Gamma(R)^{-1}$. Also (3) $\implies$ (1) is obvious. We now show $(2) \implies(3)$. Given BDF $\Gamma$-graded free modules $F = \bigoplus_{i \in I} R(-g_i)$ and $F' = \bigoplus_{j \in J} R(-h_j)$ with $\mathcal{K}_\Gamma(F) = \mathcal{K}_\Gamma(F')$, by Proposition \ref{prop: k series of a free module}, we have that $\mathcal{K}_\Gamma(F) = \sum_{g \in \Gamma} \# \{ i \in I \mid g_i = g \} \textbf{t}^g$ and $\mathcal{K}_\Gamma(F') = \sum_{g \in \Gamma} \# \{ j \in J \mid h_j = g \} \textbf{t}^g$, and hence for each $g \in \Gamma$, we have $\# \{ i \in I \mid g_i = g \} = \# \{ j \in J \mid h_j = g \}$, so we obtain a bijection between the standard bases of $F$ and $F'$, yielding $F \cong F'$. 
\end{proof}


\section{The Grothendieck Group of a Graded Ring}

For the category of finitely generated ungraded or graded modules over a fixed ring there are two typical ways to define the Grothendieck group of a commutative ring $R$ which agree in favorable circumstances, for example when the ring is a Noetherian regular local ring, or a polynomial ring in finitely many variables. One the one hand, $K_0(R)$ (or $K_{\Gamma, 0}(R)$ in the graded case) is defined to be the Grothendieck group of the commutative monoid of isomorphism classes of finitely generated projective (graded) $R$-modules with addition given by $[P] + [Q] = [P \oplus Q]$. On the other hand, $G_0(R)$ (or $G_{\Gamma, 0}(R)$ in the graded case) is defined to be the quotient of the free abelian group on the set of isomorphism classes of all finitely generated (graded) $R$-modules modulo the subgroup generated by expressions of the form $[L] - [M] + [N]$ when there exists a short exact sequence $0 \to L \to M \to N \to 0$. (See \cite{weibel2013k} for details). In this section we give analogs of these definitions for BDF $\Gamma$-graded modules and show that they agree for PDCF $\Gamma$-graded $k$-algebras whose ideal maximal graded ideal $R_+$ is generated by a regular sequence.


\subsection{The Group $K^{\text{BDF}}_{\Gamma , 0}(R)$ of a PDCF $\Gamma$-Graded $k$-Algebra}

We begin by recalling the Grothendieck group of a commutative monoid. If $Q$ is a commutative monoid, then there exists an abelian group $K$ (unique up to a unique isomorphism) and a monoid homomorphism $i : Q \to K$ which satisfies the following universal property: For all abelian groups $A$ and all monoid homomorphisms $\varphi: Q \to A$ there exists a unique group homomorphism $\tilde{\varphi} : K \to A$ such that $\varphi = \tilde{\varphi} \circ i$. Then $K$ is called the \emph{Grothendieck group} of the commutative monoid $Q$. 

One construction of the Grothendieck group of $Q$ is to define an equivalence relation $\sim$ on $Q \times Q$ by letting $(q_1, q_2) \sim (p_1, p_2)$ if and only if there exists $g \in Q$ such that $q_1 + p_2 + g = p_1 + q_2 + g$. The equivalence class of $(q, p) \in Q \times Q$ is denoted by $[q , p]$, and $K$ is defined to be the set of equivalence classes $Q \times Q / \sim$ together with the binary operation $[q , p] + [q' , p'] = [q + q' , p + p']$. 

We will need the following proposition which describes when the Grothendieck group of a commutative monoid is trivial. 

\begin{proposition}\label{prop: triviality of the grothendieck group}
Let $Q$ be a commutative monoid. Then its Grothendieck group $K$ is trivial if and only if for all $p , q \in Q$ there exists $g \in Q$ such that $p + g = q + g$. 
\end{proposition}

\begin{proof}
Since the Grothendieck group of $Q$ is unique up to a unique isomorphism we may assume that $K = Q \times Q / \sim$, and we note that the identity of $K$ is $[0 , 0]$. We have that $K$ is trivial if and only if for all $p , q \in Q$, $(p, q) \sim (0 , 0)$ which is true if and only if for all $p , q \in Q$ there exists $g \in Q$ such that $p + g = q + g$. 
\end{proof}


We define $R\text{-}\bPMod_\Gamma^\text{BDF}$ to be the category of \textbf{projective} BDF $\Gamma$-graded $R$-modules. By Theorem \ref{thm: projective implies free}, every such module is actually free. Now define $K^{\text{BDF}}_{\Gamma , 0}(R)$ to be the Grothendieck group of the commutative monoid consisting of the set of isomorphism classes of objects in $R\text{-}\bPMod_\Gamma^\text{BDF}$, together with the operation $[M] + [N] = [ M \oplus N ]$. We note that by Proposition \ref{prop: triviality of the grothendieck group} this group is not trivial because $R$ and $R \oplus R$ are both free BDF $\Gamma$-graded $R$-modules, if there existed an object $M$ of $R\text{-}\bPMod_\Gamma^\text{BDF}$ with $R \oplus M \cong R \oplus R \oplus M$ as $R$-modules then we would have $R_0 \oplus M_0 \cong R_0 \oplus R_0 \oplus M_0$ as $k$-vector spaces which is impossible as $R_0$ and $M_0$ are both finite-dimensional and $\dim_k(R_0) = 1$. Recall that the famous Eilenberg swindle shows that the Grothendieck group of the monoid consisting of \textbf{all} projective $R$-modules is trivial. We have avoided this by considering only finite-dimensionally graded modules. 

\begin{note}
The category $R\text{-}\bPMod_\Gamma^\text{BDF}$ is split exact (which can be embedded in the abelian category $R\text{-}\bMod_\Gamma^\text{BDF}$), and therefore we could have equivalently defined $K^{\text{BDF}}_{\Gamma , 0}(R)$ to be the free abelian group on the set of isomorphism classes of elements of $R\text{-}\bPMod_\Gamma^\text{BDF}$ modulo the subgroup generated by sums of the form $[L] -[M]+ [N]$ whenever there is a short exact sequence $0 \to L \to M \to N \to 0$ (see \cite{weibel2013k} Example 7.1.2). 
\end{note}

\begin{proposition}\label{prop: how to express every element of K}
Every element of $K^{\text{BDF}}_{\Gamma , 0}(R)$ can be written as $[F] - [F']$ where $F$ and $F'$ are free BDF $R$-modules. 
\end{proposition}

\begin{proof}
By definition every element of $K^{\text{BDF}}_{\Gamma , 0}(R)$ is of the form $[F_1] + \ldots + [F_n] -  [F'_1] - \ldots - [F'_m]$ for $F_i$ and $F'_j$ free BDF $R$-modules. By definition of addition in the Grothendieck group, this element is equal to $[F] - [F']$ where $F = F_1 \oplus \ldots \oplus F_n$ and $F' = F'_1 \oplus \ldots \oplus F'_m$. Again $F$ and $F'$ are free BDF $R$-modules, as desired. 
\end{proof}

\begin{theorem}\label{thm: isomorphism of projective grothendieck group and group ring}
Let $R$ be a PDCF $\Gamma$-graded $k$-algebra. There is an isomorphism $K^{\text{BDF}}_{\Gamma , 0}(R) \overset{\sim}{\to} \zz[[Q]]\{\Gamma\}$ given by sending $[F] \mapsto \mathcal{K}_\Gamma(F)$. 
\end{theorem}

\begin{proof}
By the universal property of Grothendieck groups, to define a group homomorphism $K^{\text{BDF}}_{\Gamma , 0}(R) \to \zz[[Q]]\{\Gamma\}$ it suffices to define a map from the commutative monoid of isomorphism classes of free $\Gamma$-graded BDF $R$-modules to $\zz[[Q]]\{\Gamma\}$ that respects direct sum, and by Proposition \ref{prop: additivity of Hilbert series} we have $\mathcal{K}_\Gamma(F \oplus F') = \mathcal{K}_\Gamma(F) +  \mathcal{K}_\Gamma(F')$ so such a map does indeed exist. To show injectivity, we note that if $[F] - [F'] \mapsto 0$ then $\mathcal{K}_\Gamma(F) =  \mathcal{K}_\Gamma(F')$ hence by Corollary \ref{cor: same hilbert series implies isomorphic}, $F \cong F'$ and so $[F] = [F']$ in $K^{\text{BDF}}_{\Gamma , 0}(R)$. 

For surjectivity, let $a = \sum_{g \in \Gamma} a_g \textbf{t}^g \in \zz[[Q]]\{\Gamma\}$ be such that $a_g \geq 0$ for all $g \in \Gamma$ and define $I$ to be the subset of $\Supp_\Gamma(a) \times \nn$ containing the elements $(g , n)$ such that $n \leq a_g$. Define a function $h : I \to \Gamma$ by $(g , n) \mapsto g := h_{(g , n)}$, so that we get a family $(h_i \mid i \in I)$ of elements of $\Gamma$. We claim that this is a BDF family. Indeed, for each $g \in \Gamma$, we have $\# \{ i \in I \mid h_i = g\} = a_g$ and the set $\{ h_i \mid i \in I \}$ is exactly $\Supp_\Gamma(a)$ which is finitely bounded below and downward finite by definition of $\zz[[Q]]\{\Gamma\}$. Therefore by Proposition \ref{prop: certain direct sums of BDF modules are BDF}, $F := \bigoplus_{i \in I} R(-h_i)$ is a BDF $\Gamma$-graded $R$-module. 

Then under our map $K^{\text{BDF}}_{\Gamma , 0}(R) \to \zz[[Q]]\{\Gamma\}$, we have that $[F] \mapsto \mathcal{K}_\Gamma(F)$ and the latter is equal to $\sum_{g \in \Gamma} \# \{ i \in I \mid h_i = g \} \textbf{t}^g$ by Proposition \ref{prop: k series of a free module}. Since $a = \sum_{g \in \Gamma} \# \{ i \in I \mid h_i = g\}$ as well, we have $\mathcal{K}_\Gamma(F) = a$. 
 
Surjectivity follows since every element of $\zz[[Q]]\{\Gamma\}$ can be written as $a - a'$ where $a_g , a'_g \geq 0$. 
\end{proof}


\subsection{The Group $G^{\text{BDF}}_{\Gamma , 0}(R)$ of a Graded Ring}

Define $G^{\text{BDF}}_{\Gamma , 0}(R)$ to be the free abelian group on the set of isomorphism classes of objects in $R\text{-}\bMod_\Gamma^\text{BDF}$, modulo the subgroup generated by expressions of the form $[L] -[M]+ [N]$ whenever there is a short exact sequence $0 \to L \to M \to N \to 0$. We now prove a series of lemmas which culminate in the result that $K^{\text{BDF}}_{\Gamma , 0}(R)$ and $G^{\text{BDF}}_{\Gamma , 0}(R)$ are isomorphic. We begin by defining the map that will become this isomorphism.

\begin{lemma} \label{lem: homomorphism from K to G}
Let $R$ be a PDCF $k$-algebra. There exists a group homomorphism $\Phi:  K^{\text{BDF}}_{\Gamma , 0}(R) \to G^{\text{BDF}}_{\Gamma , 0}(R)$ which sends the isomorphism class $[F]$ in $K^{\text{BDF}}_{\Gamma , 0}(R)$ of a free module $F$ to the residue of $[F]$ in the quotient $G^{\text{BDF}}_{\Gamma , 0}(R)$. 
\end{lemma}

\begin{proof}
There is a function $\varphi$ from the set of isomorphism classes of objects in $R\text{-}\bPMod_\Gamma^\text{BDF}$ to $G^{\text{BDF}}_{\Gamma , 0}(R)$ which sends an isomorphism class $[F]$ to its residue in the quotient $G^{\text{BDF}}_{\Gamma , 0}(R)$. This is well-defined because $R\text{-}\bPMod_\Gamma^\text{BDF}$ is a subcategory of $R\text{-}\bMod_\Gamma^\text{BDF}$, so if modules are isomorphic in the former, they are also isomorphic in the latter. Furthermore, it is a monoid homomorphism because given free modules $F$ and $F'$ in $R\text{-}\bPMod_\Gamma^\text{BDF}$, we have that $\varphi([F]+ [F'])= \varphi([F \oplus F']) = [F \oplus F']$ and since there is an exact sequence $0 \to F \to F \oplus F' \to F' \to 0$ in $R\text{-}\bMod_\Gamma^\text{BDF}$ we have that $[F \oplus F'] = [F] + [F']$ in $G^{\text{BDF}}_{\Gamma , 0}(R)$, and the latter is equal to $\varphi([F]) + \varphi([F'])$. Therefore the universal property of the Grothendieck group of a monoid gives the desired group homomorphism $K^{\text{BDF}}_{\Gamma , 0}(R) \to G^{\text{BDF}}_{\Gamma , 0}(R)$. 
\end{proof}

Next we give a technical lemma, followed by a lemma which allows us to express classes of modules $[M] \in G^{\text{BDF}}_{\Gamma , 0}(R)$ in terms of free modules using $\Gamma$-finite free resolutions. 

\begin{lemma}\label{lem: closure under direct sums of terms from Gamma-finite sequences}
If $\ldots \to M_2 \to M_1 \to M_0 \to 0$ is a $\Gamma$-finite sequence in $R\text{-}\bMod_\Gamma^\text{BDF}$ then for any collection of submodules $Q_n \subseteq M_n$ any subset $U \subseteq \zz_{\geq 0}$, $\bigoplus_{n \in U} Q_n$ is also in $R\text{-}\bMod_\Gamma^\text{BDF}$. 
\end{lemma}

\begin{proof}
First of all if $Q_n \subseteq M_n$ then we have that $\Supp_\Gamma(\bigoplus_{n \in U} Q_n) \subseteq \Supp_\Gamma(\bigoplus_{n \in U} M_n)$. Furthermore we also have $\Supp_\Gamma(\bigoplus_{n \in U} M_n) \subseteq \bigcup_{n = 0}^\infty \Supp_\Gamma((M_n)_g)$ and the latter is finitely bounded below and downward finite by (GF1), hence the former is as well by Lemma \ref{lem: closure of fbb and df sets under taking subsets}. Furthermore, each $Q_n$ is finite-dimensionally graded, hence for all $g \in \Gamma$, $(Q_n)_g$ is a finite-dimensional $k$-vector space. By definition $(\bigoplus_{n \in U} (Q_n)_g = \bigoplus_{n \in U} (Q_n)_g$, but by (GF2) there are only finitely many $n \in \zz_\geq 0$ for which $(M_n)_g$ is nonzero hence the same is true for $(Q_n)_g$ since $(Q_n)_g \subseteq (M_n)_g$. Therefore this direct sum is also finite-dimensional and we obtain that $\bigoplus_{n \in U} Q_n$ is finite-dimensionally graded as desired. 
\end{proof}

\begin{lemma}\label{lem: expression of every BDF module as a sum of frees}
If $M$ is a BDF $\Gamma$-graded $R$-module and $F_\bullet$ is a $\Gamma$-finite free resolution of $M$, then in $G^{\text{BDF}}_{\Gamma , 0}(R)$ we have $[M] = [ \bigoplus_{n = 0}^\infty F_{2n}]- [\bigoplus_{n = 0}^\infty F_{2n + 1}]$.
\end{lemma}

\begin{proof}
If $\ldots \to F_1 \overset{\varphi_1}{\to} F_0 \overset{\varphi_0}{\to} M \to 0$ is a $\Gamma$-finite free resolution of $M$, consider the map $\varphi := \bigoplus_{n = 0}^\infty \varphi_{2n} : \bigoplus_{n = 0}^\infty F_{2n} \to M \oplus \bigoplus_{n = 0}^\infty F_{2n + 1}$. Note that both the domain and codomain of $\varphi$ are BDF $\Gamma$-graded $R$-modules by Lemma \ref{lem: closure under direct sums of terms from Gamma-finite sequences}. The kernel of this map is $\bigoplus_{n = 0}^\infty \ker \varphi_{2n}$ and the cokernel is $\bigoplus_{n = 0}^\infty F_{2n + 1} / \im \varphi_{2n + 2}$. Therefore we have the following exact sequence. 
\begin{equation*}
0 \to \bigoplus_{n = 0}^\infty \ker \varphi_{2n} \to \bigoplus_{n = 0}^\infty F_{2n} \overset{\varphi}{\to} M \oplus \bigoplus_{n = 0}^\infty F_{2n + 1} \to \bigoplus_{n = 0}^\infty F_{2n + 1} / \im \varphi_{2n + 2} \to 0
\end{equation*}
So in $G^{\text{BDF}}_{\Gamma , 0}(R)$ we have 
\begin{equation}\label{eq: grothendieck group sum}
[\bigoplus_{n = 0}^\infty F_{2n + 1} / \im \varphi_{2n + 2}] - [M \oplus \bigoplus_{n = 0}^\infty F_{2n + 1}] + [ \bigoplus_{n = 0}^\infty F_{2n}] - [\bigoplus_{n = 0}^\infty \ker \varphi_{2n}] = 0. 
\end{equation}
However, we compute: 
\begin{equation*}
\bigoplus_{n = 0}^\infty F_{2n + 1} / \im \varphi_{2n + 2} = \bigoplus_{n = 0}^\infty F_{2n + 1} / \ker \varphi_{2n + 1} \cong \bigoplus_{n = 0}^\infty \im \varphi_{2n + 1} = \bigoplus_{n = 0}^\infty \ker \varphi_{2n}. 
\end{equation*}
Therefore the first and last terms of Equation \ref{eq: grothendieck group sum} cancel giving $[M \oplus \bigoplus_{n = 0}^\infty F_{2n + 1}] = [ \bigoplus_{n = 0}^\infty F_{2n}]$. The LHS of this equation is equal to $[M] \oplus [\bigoplus_{n = 0}^\infty F_{2n + 1}]$, therefore subtraction yields $[M] =  [ \bigoplus_{n = 0}^\infty F_{2n}]- [\bigoplus_{n = 0}^\infty F_{2n + 1}]$. 

\end{proof}

Finally we arrive at the isomorphism between the two Grothendieck groups. 

\begin{theorem}\label{thm: two defs of grothendieck groups are the same}
Let $R$ be a PDCF $\Gamma$-graded $k$-algebra with the property that $R_+$ is generated by a homogeneous regular sequence. The group homomorphism $\Phi: K^{\text{BDF}}_{\Gamma , 0}(R) \to G^{\text{BDF}}_{\Gamma , 0}(R)$ described in Lemma \ref{lem: homomorphism from K to G} is an isomorphism. 
\end{theorem}

\begin{proof}
Given $[M] \in G^{\text{BDF}}_{\Gamma , 0}(R)$, by the BDF Hilbert Syzygy Theorem (Theorem \ref{thm: BDF modules have gamma finite free resolutions}), there exists a BDF $\Gamma$-finite free resolution $\ldots \to F_1 \to F_0 \to M \to 0$ of $M$, and then by Lemma \ref{lem: expression of every BDF module as a sum of frees} we have that $[M] =  [F] - [F']$ in $G^{\text{BDF}}_{\Gamma , 0}(R)$, where $F = \bigoplus_{n = 0}^\infty F_{2n}$ and $F' = \bigoplus_{n = 0}^\infty F_{2n + 1}$. Since $F$ and $F'$ are both free, $[F] - [F']$ is an element of $K^{\text{BDF}}_{\Gamma , 0}(R)$ and we have that $\Phi(  [ F ] - [F'] ) = [M]$ so $\Phi$ is surjective. 

On the other hand, because $K$-series are additive (Corollary \ref{prop: additivity of Hilbert series}), the function which sends the isomorphism class of a BDF $\Gamma$-graded $R$-module $M$ to its $K$-series $\mathcal{K}_\Gamma(M)$ defines, via the universal property of $G^{\text{BDF}}_{\Gamma , 0}(R)$, a group homomorphism $\Psi: G^{\text{BDF}}_{\Gamma , 0}(R) \to \zz[[Q]]\{\Gamma\}$. Composing this with the isomorphism $\Theta: \zz[[Q]]\{\Gamma\} \to K^{\text{BDF}}_{\Gamma , 0}(R)$ of Theorem \ref{thm: isomorphism of projective grothendieck group and group ring}, we obtain a group homomorphism $G^{\text{BDF}}_{\Gamma , 0}(R) \to K^{\text{BDF}}_{\Gamma , 0}(R)$ which is a left-inverse of $\Phi$, since the composition $\Theta \circ \Psi \circ \Phi$ sends $[F] \in K^{\text{BDF}}_{\Gamma , 0}(R)$ to $[F] \in G^{\text{BDF}}_{\Gamma , 0}(R)$ to $\mathcal{K}_\Gamma([F]) \in \zz[[Q]]\{\Gamma\}$ back to $[F] \in K^{\text{BDF}}_{\Gamma , 0}(R)$. Since $\Phi$ has a left inverse, it is is injective. 
\end{proof}

\begin{definition}
We call the group $K^{\text{BDF}}_{\Gamma , 0}(R) \cong G^{\text{BDF}}_{\Gamma , 0}(R)$ the \emph{BDF $\Gamma$-graded Grothendieck group} of $R$. 
\end{definition}


\section{Ring Structure on the Grothendieck Group}

In this section we define the product which makes the Grothendieck group into a ring. We begin with two preliminary results which lead to the definition of this product on $K^{\text{BDF}}_{\Gamma , 0}(R)$, then we give a BDF version of Serre's formula to define the product on $G^{\text{BDF}}_{\Gamma , 0}(R)$.

\begin{lemma}\label{lem: sum of bdf families is bdf}
If $(g_i \mid i \in I)$ and $(h_j \mid j \in J)$ are BDF families of elements of $\Gamma$ then $( g_i + h_j \mid (i , j) \in I \times J )$ is also a BDF family. 
\end{lemma}

\begin{proof}
Let $U = \{g_i \mid i \in I\}$ and $V = \{h_j \mid j \in J\}$. We have that $\{ g_i + h_j \mid (i , j) \in I \times J \} = U + V$ and because $U$ and $V$ are finitely bounded below and downward finite by hypothesis, by Lemma \ref{lem: sum of fbb df is too} the LHS is too. Given any $c \in \Gamma$ the set $T = \{ (g , h) \in U \times V \mid g + h = c \}$ is finite by Lemma \ref{lem: finitely bounded below and downward finite sets have the finite sum property}. Then we compute: 
\begin{align*}
\{ (i , j) \in I \times J \mid g_i + h_j = c \} &= \bigcup_{(g , h) \in T} \{ (i , j) \in I \times J \mid (g_i , h_j ) = (g , h) \} 
\\&= \bigcup_{(g , h) \in T} \{ i \in I \mid g_i = g \} \times \{ j \in J \mid h_j = h \}. 
\end{align*}
By hypothesis, the sets $\{ i \in I \mid g_i = g \}$ and $\{ j \in J \mid h_j = h \}$ are finite, and since $T$ is finite, this union is finite as well. 
\end{proof}

\begin{proposition}
If $F$ and $F'$ are free BDF $\Gamma$-graded $R$-modules then $\mathcal{K}_\Gamma(F \otimes_R F') = \mathcal{K}_\Gamma(F) \mathcal{K}_\Gamma(F')$
\end{proposition}

\begin{proof}
If $F = \bigoplus_{i \in I} R(-g_i)$ and $F = \bigoplus_{j \in J} R(-h_j)$, then because tensor product commutes with direct sum, we can compute 
\begin{equation*}
F \otimes_R F' = \left( \bigoplus_{i \in I} R(-g_i) \right) \otimes_R \left( \bigoplus_{j \in J} R(-h_j) \right) \cong \bigoplus_{i \in I , j \in J} R(-g_i) \otimes_R R(-h_j) \cong \bigoplus_{i \in I , j \in J} R(-(g_i + h_j))
\end{equation*}

The family $( g_i + h_j \mid (i , j) \in I \times J )$ of elements of $\Gamma$ is BDF by Lemma \ref{lem: sum of bdf families is bdf}, so by Proposition \ref{prop: k series of a free module}, we have that $\mathcal{K}_\Gamma( F \otimes_R F' ) = \sum_{g \in \Gamma} \# \{ (i, j) \in I \times J \mid g_i + h_j = g \} \textbf{t}^g$. On the other hand, again by Proposition \ref{prop: k series of a free module} we also have that $\mathcal{K}_\Gamma( F) = \sum_{g \in \Gamma} \# \{ i \in I \mid g_i = g \} \textbf{t}^g$ and that $\mathcal{K}_\Gamma( F' ) = \sum_{g \in \Gamma} \# \{ j \in J \mid h_j = g \} \textbf{t}^g$, so by definition $\mathcal{K}_\Gamma(F) \mathcal{K}_\Gamma(F') = \sum_{g \in \Gamma} \left( \sum_{a + b = g} \# \{ j \in J \mid g_i = a \} \cdot \# \{ j \in J \mid h_j = b \} \right) \textbf{t}^g$. But there is an obvious bijection $\{ (i , j) \in I \times J \mid g_i + h_j = g \} \leftrightarrow \bigsqcup_{a + b = g} \{ i \in I \mid g_i = a \} \times \{ j \in J \mid h_j = b \}$, so we obtain the desired result.   
\end{proof}

\begin{corollary}\label{cor: grothendieck ring}
The binary operation $[F] , [F'] \mapsto [F \otimes_R F']$ gives a well-defined commutative ring structure on $K^{\text{BDF}}_{\Gamma , 0}(R)$ with multiplicative identity equal to $[R]$ and the isomorphism $K^{\text{BDF}}_{\Gamma , 0}(R) \to \zz[[Q]]\{\Gamma \}$ given in Theorem \ref{thm: isomorphism of projective grothendieck group and group ring} is a ring map. 
\end{corollary}

\begin{proof}
In fact because the map $K^{\text{BDF}}_{\Gamma , 0}(R) \to \zz[[Q]]\{\Gamma \}$ given by $K$-series is a group isomorphism, it follows that this binary operation is well-defined and indeed gives a commutative ring structure on $K^{\text{BDF}}_{\Gamma , 0}(R)$. Since $\mathcal{K}_\Gamma(R) = 1$, $[R]$ is the multiplicative identity of $K^{\text{BDF}}_{\Gamma , 0}(R)$. 
\end{proof}


\subsection{BDF Serre's Formula}

To make sense of Serre's formula in the BDF setting, we must place a linear topology on $G^{\text{BDF}}_{\Gamma , 0}(R)$ that corresponds to the $Z$-adic topology on $\zz[[Q]]\{ \Gamma \}$. Given a finite set $U \subseteq \Gamma$, define $Z^G_U$ to be the subgroup of $G^{\text{BDF}}_{\Gamma , 0}(R)$ generated by the set of all $[M] \in G^{\text{BDF}}_{\Gamma , 0}(R)$ such that $\Supp_\Gamma(M) \cap [U , -\infty) = \emptyset$. This gives a directed filtration $Z^G$ on $G^{\text{BDF}}_{\Gamma , 0}(R)$ which corresponds exactly to the $Z$ filtration on $\zz[[Q]]\{ \Gamma \}$ under the isomorphism $G^{\text{BDF}}_{\Gamma , 0}(R) \to \zz[[Q]]\{ \Gamma \}$. Hence the $Z^G$-adic topology is also Hausdorff so all series converge to at most one element. 

If $R$ is a Noetherian ring, it is well-known that if $M_\bullet$ is a \textbf{bounded} chain complex of finitely generated modules, then the Euler characteristic of the complex in $G_0(R)$ depends only on its homology, i.e. $\sum_i (-1)^i [M_i] = \sum_i (-1)^i [H_i(M)]$ (see \cite{weibel2013k}, Chapter II, Proposition 6.6). The following lemma is a BDF version of that result for \textbf{$\Gamma$-finite} complexes. 

\begin{lemma}\label{lem: convergence of sequences in G}
If $\ldots \to M_2 \overset{\varphi_2}{\to} M_1 \overset{\varphi_1}{\to} M_0 \overset{\varphi_0}{\to} 0$ is a $\Gamma$-finite chain complex of BDF $\Gamma$-graded $R$-modules then $\sum_{i = 0}^\infty (-1)^i [M_i]$ and $\sum_{i = 0}^\infty (-1)^i [H_i(M_\bullet)]$ both converge to $[\bigoplus_{i = 0}^\infty M_{2i}] - [\bigoplus_{i = 0}^\infty M_{2i + 1}]$ in the $Z^G$-adic topology on $G^{\text{BDF}}_{\Gamma , 0}(R)$.
\end{lemma}

\begin{proof}
First of all, because the sequence is $\Gamma$-finite, in particular $V := \bigcup_{i = 0}^\infty \Supp_\Gamma(M_i) \subseteq \Gamma$ is finitely bounded below and downward finite. Hence given $U \in \mathcal{P}_\text{fin}(\Gamma)$, there are only finitely many $g \in (-\infty , U] \cap V$. Therefore again since the sequence is $\Gamma$-finite, there exists $n \in \nn$ such that $(M_i)_g = 0$ for all $g \in (-\infty , U]$ and $i \geq n$. We may assume without loss of generality that $i$ is odd, the even case being similar. In $G^{\text{BDF}}_{\Gamma , 0}(R)$ the following holds. 
\begin{equation*}
[\bigoplus_{j = 0}^\infty M_{2j}] - [\bigoplus_{j = 0}^\infty M_{2j + 1}] - \sum_{j = 0}^i (-1)^j [M_j] = [\bigoplus_{j = (i + 1)/2}^\infty M_{2j}] - [\bigoplus_{j = (i + 1) / 2 }^\infty M_{2j + 1}]
\end{equation*}

However, since $i \geq n$, it follows that $\left(\bigoplus_{j = (i + 1)/2}^\infty M_{2j}\right)_g = \bigoplus_{j = (i + 1)/2}^\infty (M_{2j})_g = 0$, hence $\bigoplus_{j = (i + 1)/2}^\infty M_{2j} \in Z^G_U$. Similarly $\bigoplus_{j = (i + 1) / 2 }^\infty M_{2j + 1} \in Z^G_U$, so their difference is in $Z^G_U$ as desired. 

Now given $U \in \mathcal{P}_\text{fin}(\Gamma)$, if we choose $n$ such that $(M_i)_g = 0$ for all $g \in (-\infty , U]$ and $m \geq n$, then we compute:
\begin{align*}
\sum_{j = 0}^m (-1)^j [H_j(M_\bullet)] &= \sum_{j = 0}^m (-1)^j [\ker \varphi_j / \im \varphi_{j + 1}] = \sum_{j = 0}^m (-1)^j ( [ \ker \varphi_j ] - [\im \varphi_{j + 1}] ) 
\\&= \sum_{j = 0}^m (-1)^j ( [ \ker \varphi_j ] - [M_{j + 1} / \ker \varphi_{j + 1}] ) = \sum_{j = 0}^m (-1)^j ( [\ker \varphi_j] - [M_{j + 1}] + [\ker \varphi_{j + 1}] ) 
\\&\overset{(1)}{=} [M_0] - [M_1] + [M_2] - \ldots - (-1)^m [M_{m + 1}] + (-1)^m [\ker \varphi_{m + 1}] 
\\&= [M_0] - [M_1] + [M_2] - \ldots - (-1)^m [\im \varphi_{m + 1}] 
\end{align*}
Here (1) follows from cancelation and the fact that $\ker \varphi_0 = M_0$. Therefore we have that $\sum_{j = 0}^m (-1)^j [M_j] - \sum_{j = 0}^m (-1)^j [H_j(M_\bullet)] = (-1)^m [\im \varphi_{m + 1}]$. But because $m \geq n$ it follows that $(M_m)_g = 0$, and hence $(\im \varphi_{m + 1})_g = 0$, for all $g \in (-\infty , U]$, so $[\im \varphi_{m + 1}] \in Z^G_U$. Therefore the difference of the series $\sum_{i = 0}^\infty (-1)^i [M_i]$ and $\sum_{i = 0}^\infty (-1)^i [H_i(M_\bullet)]$ converges to $0$, so they converge to the same element in the $Z^G_U$-adic topology. 
\end{proof}

\begin{lemma}
Let $M$ and $N$ be BDF $\Gamma$-graded $R$-modules. If $F_\bullet$ and $G_\bullet$ are $\Gamma$-finite BDF free resolutions of $M$ and $N$ respectively, then $\sum_{i = 0}^\infty (-1)^i [ \Tor_i(M , N)]$ converges to $[\bigoplus_{i , j = 0}^\infty (F_{2i } \otimes_R G_{2j}) \oplus (F_{2i + 1} \otimes_R G_{2j + 1}) ] - [\bigoplus_{i , j = 0}^\infty (F_{2i + 1} \otimes_R G_{jm}) \oplus (F_{2i} \otimes_R G_{2j + 1})]$ in the $Z^G$-adic topology on $G^{\text{BDF}}_{\Gamma , 0}(R)$. 
\end{lemma}

\begin{proof}
We form the double complex $(F_\bullet \to M \to 0) \otimes_R (G_\bullet \to N \to 0)$ as in Figure \ref{fig: tenor product of complexes} and note, as in the proof of Proposition \ref{prop: symmetry of graded tor}, each column of this complex other than the right-most is exact, and is therefore a free resolution of its bottom term. Taking the direct sum of all columns with bottom term $F_{2i} \otimes_R N$ we obtain an exact sequence $\ldots \to \bigoplus_{i = 0}^\infty F_{2i} \otimes_R G_1 \to \bigoplus_{i = 0}^\infty F_{2i} \otimes_R G_0 \to \bigoplus_{i = 0}^\infty F_{2i} \otimes_R N \to 0$ whose terms are BDF by Lemma \ref{lem: closure under direct sums of terms from Gamma-finite sequences} because the rows of the double complex are BDF $\Gamma$-finite, and which is $\Gamma$-finite itself because by Lemma \ref{lem: tensor product of complexes is gamma finite}, the double complex is $\Gamma$-finite. Similarly the direct sum of the odd columns also gives a $\Gamma$-finite exact sequence $\ldots \to \bigoplus_{i = 0}^\infty F_{2i + 1} \otimes_R G_1 \to \bigoplus_{i = 0}^\infty F_{2i + 1} \otimes_R G_0 \to \bigoplus_{i = 0}^\infty F_{2i + 1} \otimes_R N \to 0$. Since these are in fact free resolutions, by Lemma \ref{lem: expression of every BDF module as a sum of frees} we have $[\bigoplus_{i = 0}^\infty F_{2i} \otimes_R N] = [\bigoplus_{i, j = 0}^\infty F_{2i} \otimes_R G_{2j} ] - [\bigoplus_{i, j = 0}^\infty F_{2i} \otimes_R G_{2j + 1}]$ and similarly $[\bigoplus_{i = 0}^\infty F_{2i + 1} \otimes_R N] = [\bigoplus_{i, j = 0}^\infty F_{2i + 1} \otimes_R G_{2j} ] - [\bigoplus_{i, j = 0}^\infty F_{2i + 1} \otimes_R G_{2j + 1}]$. Therefore by Lemma \ref{lem: convergence of sequences in G} we can compute as follows.   
\begin{align*}
&\sum_{i = 0}^\infty (-1)^i [ \Tor_i(M , N)] = \sum_{i = 0}^\infty (-1)^i [H_i(F_\bullet \otimes_R N)] 
\\&= \sum_{i = 0}^\infty (-1)^i [ F_i \otimes_R N ]
\\&= [\bigoplus_{i = 0}^\infty F_{2i} \otimes_R N] - [\bigoplus_{i = 0}^\infty F_{2i + 1} \otimes_R N]
\\&= \left( [\bigoplus_{i, j = 0}^\infty F_{2i} \otimes_R G_{2j} ] - [\bigoplus_{i, j = 0}^\infty F_{2i} \otimes_R G_{2j + 1}]  \right)  - \left( [\bigoplus_{i, j = 0}^\infty F_{2i + 1} \otimes_R G_{2j} ] - [\bigoplus_{i, j = 0}^\infty F_{2i + 1} \otimes_R G_{2j + 1}] \right)
\\&= [ \bigoplus_{i , j = 0}^\infty (F_{2i } \otimes_R G_{2j}) \oplus (F_{2i + 1} \otimes_R G_{2j + 1}) ] - [\bigoplus_{i , j = 0}^\infty (F_{2i + 1} \otimes_R G_{jm}) \oplus (F_{2i} \otimes_R G_{2j + 1})]. 
\end{align*}
\end{proof}

\begin{corollary}[BDF Serre's Formula]\label{cor: ring structure on G}
The binary operation $[M] \cdot [N] = \sum_{i = 0}^\infty (-1)^i [\Tor_i(M , N)]$ gives a well-defined commutative ring structure on $G^{\text{BDF}}_{\Gamma , 0}(R)$ with multiplicative identity equal to $[R]$ and the isomorphism $K^{\text{BDF}}_{\Gamma , 0}(R) \overset{\sim}{\to} G^{\text{BDF}}_{\Gamma , 0}(R)$ given in Theorem \ref{thm: two defs of grothendieck groups are the same} is a ring map.  
\end{corollary}

\begin{proof}
Under the isomorphism $\Theta \circ \Psi: G^{\text{BDF}}_{\Gamma , 0}(R) \to K^{\text{BDF}}_{\Gamma , 0}(R)$ given in the proof of Theorem \ref{thm: two defs of grothendieck groups are the same}, we have that $\sum_{i = 0}^\infty (-1)^i [\Tor_i(M , N)] = [\bigoplus_{i , j = 0}^\infty (F_{2i} \otimes_R G_{2j}) \oplus (F_{2i + 1} \otimes_R G_{2j + 1}) ] - [\bigoplus_{i , j = 0}^\infty (F_{2i + 1} \otimes_R G_{jm}) \oplus (F_{2i} \otimes_R G_{2j + 1})]$ maps to the class in $K^{\text{BDF}}_{\Gamma , 0}(R)$ with the same representatives. But in $K^{\text{BDF}}_{\Gamma , 0}(R)$, using the ring structure defined in Corollary \ref{cor: grothendieck ring}, we compute: 
\begin{align*}
&[\bigoplus_{i , j = 0}^\infty (F_{2i} \otimes_R G_{2j}) \oplus (F_{2i + 1} \otimes_R G_{2j + 1}) ] - [\bigoplus_{i , j = 0}^\infty (F_{2i + 1} \otimes_R G_{2j}) \oplus (F_{2i} \otimes_R G_{2j + 1})] 
\\&= \left( [\bigoplus_{i = 0}^\infty F_{2i}] - [\bigoplus_{i = 0}^\infty F_{2i + 1}] \right) \cdot \left( [\bigoplus_{i = 0}^\infty G_{2i}] - [\bigoplus_{i = 0}^\infty G_{2i + 1}] \right)
\end{align*}
The latter is product of the images of $[M]$ and $[N]$ under $\Theta \circ \Psi$. Since this is an isomorphism, the binary operation $[M] \cdot [N] = \sum_{i = 0}^\infty (-1)^i [\Tor_i(M , N)]$ does indeed give a ring structure on $G^{\text{BDF}}_{\Gamma , 0}(R)$ and $\Theta \circ \Psi$ is an isomorphism of rings. 
\end{proof}

\begin{definition}
$K^{\text{BDF}}_{\Gamma , 0}(R)$ together with the ring structure given by Corollary \ref{cor: grothendieck ring}, or equivalently $G^{\text{BDF}}_{\Gamma , 0}(R)$ together with the ring structure given by Corollary \ref{cor: ring structure on G}, is called the \emph{BDF $\Gamma$-graded Grothendieck ring} of $R$.
\end{definition}


\bibliographystyle{alpha}

\bibliography{The_Grothendieck_Ring_of_Certain_Non-Noetherian_Group-Graded_Algebras_via_K-Series}

\end{document}